\documentclass[11pt,a4paper,twoside]{article}

\usepackage[T1]{fontenc}
\usepackage[cp1250]{inputenc} 
\usepackage{times}

\usepackage[leqno]{amsmath} 
\usepackage{amsthm,amsfonts,amssymb} 
\usepackage{bm} 

\input xy  \xyoption{all}

\usepackage{geometry}
\usepackage{color}

\usepackage{enumerate} 
\usepackage{hyperref} 

\newcommand{\R}{\mathbb{R}}
\newcommand{\Z}{\mathbb{Z}}

\newcommand{\wt}[1]{\widetilde{#1}} 

\newcommand{\und}{\underline}

\newcommand{\lra}{\longrightarrow}
\newcommand{\ra}{\rightarrow}

\newcommand{\eps}{\epsilon}
\newcommand{\veps}{\varepsilon}
\newcommand{\del}{\delta}

\newcommand{\pa}{\partial}
\newcommand{\Sec}{\operatorname{Sec}}
\newcommand{\id}{\operatorname{id}}

\newcommand{\PP}{\operatorname{P}} 

\newcommand{\dd}{\operatorname{d}}

\newcommand{\thh}[1]{{#1}^{\mathrm{th}}}
\newcommand{\st}[1]{{#1}^{\mathrm{st}}}

\newcommand{\T}{\mathrm{T}} 
\newcommand{\TT}[1]{\mathrm{T}^{#1}} 
\newcommand{\Thol}[1]{\mathrm{\wt T}^{#1,#1}} 
\newcommand{\Tt}[1]{\mathrm{T}^{(#1)}} 
\newcommand{\tgT}{\T} 
\newcommand{\jet}[1]{\bm{\mathrm{t}}^{#1}} 
\newcommand{\E}[1]{E^{#1}} 
\newcommand{\F}{\mathrm{F}}  
\newcommand{\M}{\mathrm{M}}  
 \newcommand{\Ttsemihol}[1]{\wt \T^{(#1)} }
 \newcommand{\Tsemihol}[1]{\wt \T^{#1} }

\newcommand{\Adm}{\mathcal{ADM}} 
\newcommand{\Pp}{\mathcal{P}} 
\newcommand{\Weil}{\mathbb{D}\, } 
\newcommand{\momenta}{\upsilon} 
\newcommand{\tauM}[2]{\tau^{#1}_{#2}} 
\newcommand{\tauE}[2]{\tau^{#1}_{#2}} 

\newcommand{\OmegaALL}{\Omega^{\bullet}} 
\newdir{|>}{!/4,5pt/@{|}*:(1,-.2)@^{>}*:(1,+.2)@_{>}}
\def\relto{\rightarrow\!\!\vartriangleright} 

\def\<#1>{\left\langle #1\right\rangle}
\def\(#1){\left( #1\right)}

\numberwithin{equation}{section} 

\theoremstyle{plain}
\newtheorem{thm}{Theorem}[section]
\newtheorem{prop}[thm]{Proposition}
\newtheorem{lem}[thm]{Lemma}

\theoremstyle{definition}
\newtheorem{df}[thm]{Definition}

\newtheorem{problem}[thm]{Problem}

\theoremstyle{remark}
\newtheorem{rem}[thm]{Remark}
\newtheorem{cor}[thm]{Corollary}

\begin{document}

\title{Bundle-theoretic methods for higher-order variational calculus\footnote{This research was supported by Polish National Science Center grant
under the contract number DEC-2012/06/A/ST1/00256.}}
\author{Micha\l\ J\'{o}\'{z}wikowski\footnote{\emph{Institute of Mathematics. Polish Academy of Sciences} (email: \texttt{mjozwikowski@gmail.com})},
Miko\l aj Rotkiewicz\footnote{\emph{Institute of Mathematics. Polish Academy of Sciences} and \emph{Faculty of Mathematics, Informatics and Mechanics. University of Warsaw} (email: \texttt{mrotkiew@mimuw.edu.pl})}}

\maketitle
\begin{abstract}
We present a geometric interpretation of the integration-by-parts formula on an arbitrary vector bundle. As an application we give a new geometric formulation of higher-order variational calculus.
\end{abstract}





\section{Introduction}

\paragraph{Our results.}
The main motivation of this paper is to clarify
the geometry of higher-order variational calculus.
In order to study this topic we had to introduce new geometric tools and prove some results which we quickly sketch below.

First, we observe that, given a vector bundle $\sigma:E\ra M$, it is possible to characterize two canonical morphisms $\Upsilon_{k,\sigma}:\Tsemihol{k,k}\sigma\ra\sigma$ and $\momenta_{k,\sigma}:\Tsemihol{k,k}\sigma\ra\TT k\sigma$, where \emph{the bundle of semi-holonomic vectors} $\Tsemihol{k,k}\sigma$ consists of all elements of $\TT{k}\TT{k}E$ which project to the elements of $\TT{2k}M\subset \TT{k}\TT{k}M$ (i.e., to \emph{holonomic} vectors). Morphisms $\Upsilon_{k,\sigma}$ and $\momenta_{k,\sigma}$ provide an elegant geometric description of the $\thh{k}$
order integration-by-parts formula on the bundle $\sigma$. A more precise formulation is provided by Theorem \ref{thm:k_hol} below.

Then, we use this result to obtain a comprehensive and, to certain extent, simpler geometric interpretation of the standard procedure of deriving both the Euler-Lagrange equations (forces) and the boundary terms (generalized momenta) for a $\thh{k}$ order variational problem. Our description is ``natural'' in the sense that it mimics the standard way of deriving the Euler-Lagrange equations.
We start from a homotopy $\gamma(t,s)$ in $M$ and then transform the variation of the action $\jet 1_{s=0}\int_{t_0}^{t_1} L(\jet k_t \gamma(t,s))\dd t$ to extract the integral and the boundary terms. Here $\jet k_t \gamma$ stands for the $\thh{k}$ jet of a curve $\gamma=\gamma(t)$ at a point $t\in\R$. In particular, $\jet 1_{t=t_0} \gamma(t)$ is the tangent vector to the curve $\gamma$ at $t_0$.  On the computational level this procedure consists basically of two steps:
\begin{itemize}
\item reversing the order of differentiation to get $\jet k_t\jet 1_s \gamma(t,s)$ from $\jet 1_s\jet k_t\gamma(t,s)$,
\item performing the $\thh{k}$ order integration by parts to extract $\jet 1_{s=0}\gamma(t,s)$ from $\jet k_t\jet 1_{s=0}\gamma(t,s)$.
\end{itemize}
The geometric interpretation of the first step is the following. Vector $\jet 1_s\jet k_t\gamma(t,s)$ is the image of $\jet k_t\jet 1_s\gamma(t,s)$ under the canonical flip $\kappa_{k}:\TT k\T M\ra \T\TT k M$, hence reversing this operation requires applying the dual map $\veps_{k}:\TT\ast\TT kM\ra \TT k\TT\ast M$ to the differential of the Lagrangian \cite{Cantrin_Crampin_Inn_can_isom_1989}. Concerning the second, we can apply  the maps $\Upsilon_{k,\tau^\ast}$ and $\momenta_{k-1,\tau^\ast}$ introduced by us (here $\tau^\ast:\TT\ast M\ra M$ stands for the cotangent fibration). In this way we obtain the geometric formula $\Upsilon_{k,\tau^\ast}\left(\jet k\veps_{k}\left(\dd L(\jet k\gamma(t))\right)\right)$ describing the force (integral term) along the trajectory $\gamma(t)=\gamma(t, 0)$. Thus the Euler-Lagrange equations along $\gamma(t)$ read as
$$\Upsilon_{k,\tau^\ast}\left(\jet k_t\veps_{k}\left(\dd L(\jet k\gamma(t))\right)\right)=0.$$
The geometric formula for the momentum (boundary term)
is obtained in a similar way using the map $\momenta_{k-1,\tau^\ast}$. The precise description is provided by Theorems \ref{thm:var_calc} and \ref{thm:EL}.

\paragraph{State of research, applications.}
In the theory of jet bundles there exists a well established notion of \emph{semi-holonomic jets} \cite{Saunders_Geom_of_jet_bndls}. These objects share some similarities with our \emph{semi-holonomic vectors}, however, in principle, are different, so that the similarity of names shall not be confusing.
Morphisms $\Upsilon_{k,\sigma}$ and $\momenta_{k,\sigma}$ were, to our best knowledge, so far not present in the literature.
On the contrary, the problem of geometric formulation of higher-order variational calculus has gained a lot of interest in the mathematical community and has many solutions (e.g., \cite{Crampin_EL_higher_order_1990,Leon_Lacomb_lagr_sbmfd_ho_mech_sys_1989, Tulcz_ehres_jet_theory_2006,Viatgliano_lagr_ham_h_field_theor_2010}), briefly reviewed in Section \ref{sec:examples}.

Since the geometry of higher-order variational calculus is a well-established topic, one should ask about other applications of morphisms $\Upsilon_{k,\sigma}$ and $\momenta_{k,\sigma}$. We believe that our result makes it easier to generalize higher-order variational calculus and mechanics to the framework of algebroids. Apart from some very special cases (see \cite{Gay_Holm_Inn_inv_ho_var_probl_2012}), we are not aware of any such generalization in the spirit of the classical papers quoted above. Our preprint \cite{MJ_MR_higher_algebroids_2013} contains some results in this direction, which make use of the tools introduced here. We present them briefly in Section \ref{sec:examples}.

\paragraph{Outline of the paper.}
In the preliminary Section \ref{sec:perliminaries} we revise basic information on higher tangent bundles, vector bundles and their lifts, as well as canonical pairings between such objects. We also recall the notion of the canonical flip $\kappa_k:\TT k\T M\lra\T\TT kM$ and its dual $\veps_k:\TT\ast\TT kM\ra\TT k\TT\ast M$.

Section \ref{sec:lemma}, with its central Theorem \ref{thm:k_hol}, contains our main results. We begin by introducing the notion of the bundle of semi-holonomic vectors $\Tsemihol{n_1,\hdots,n_k}\sigma$ and the canonical projection $\PP_k:\Ttsemihol k \sigma := \Tsemihol{1,\ldots, 1}\sigma \ra\TT k\sigma$. In Theorem \ref{thm:k_hol} we provide a geometric construction of the canonical maps $\Upsilon_{k,\sigma}$ and $\momenta_{k,\sigma}$ which give a comprehensive geometric interpretation of the integration-by-parts procedure on a vector bundle $\sigma$. We also state Lemma \ref{lem:Upsilon_canonical} about the universality of the map $\Upsilon_{k,\sigma}$, whose proof is postponed to the Appendix.

In Section \ref{sec:EL} we show how to apply our results to higher-order variational problems (Problem \ref{prob:var_M}). In particular, in Theorem \ref{thm:var_calc}, we obtain the general formula for the variation of the action including the forces and the generalized momenta along the trajectory. As a corollary we prove Theorem \ref{thm:EL}, which characterizes the extremals of Problem \ref{prob:var_M} in terms of higher-order Euler-Lagrange equations and general transversality conditions. We also give a geometric and local description of the Euler-Lagrange equations and generalized momenta.

Finally, in Section \ref{sec:examples}, we briefly  discuss different approaches to the geometry of higher-order variational calculus and relate our work in this direction with the papers of Tulczyjew \cite{Tulcz_diff_lagr_1975,Tulcz_ehres_jet_theory_2006}. Later we sketch some result from \cite{MJ_MR_higher_algebroids_2013}, which show an application of our results to higher-order variational problems on algebroids. We end with an example of Riemannian cubic polynomials.


\newpage
\section{Preliminaries}\label{sec:perliminaries}\label{sec:notation}

\paragraph{Higher tangent bundles.}
Throughout the paper we shall work with \emph{higher tangent bundles} and use the standard notation  $\TT k M$ for the \emph{$\thh{k}$ tangent bundle} of the manifold $M$. Points in the total space of this bundle will be called \emph{$k$--velocities}. An element represented by a curve $\gamma:[t_0,t_1]\ra M$ at $t\in[t_0,t_1]$ will be denoted by $\jet k\gamma(t)$ or $\jet k_t\gamma(t)$.

The $\st{k+1}$ tangent bundle is canonically included in the tangent space of the $\thh{k}$ tangent bundle (see, e.g., \cite{Tulcz_lagr_diff_1976}):
$$\iota^{1,k}:\TT{k+1}M\subset\T\TT k M, \quad \jet {k+1}_{t=0}\gamma(t)\longmapsto \jet 1_{t=0}\jet k_{s=0} \gamma(t+s).$$
The composition of this injection with the canonical projection $\tau_{\TT k M}:\T\TT kM\ra\TT kM$ defines the structure of the
\emph{tower of higher tangent bundles}:
$$\TT k M\lra\TT {k-1} M\lra \TT {k-2} M\lra\hdots\lra\T M\lra M.$$
 The canonical projections from higher to lower-order tangent bundles will be denoted by $\tauM ks:\TT k M\ra\TT {s} M$ (for $k\geq s$). Instead of $\tauM k0:\TT kM\ra M$ we simply write  $\tau^k$ and, instead of $\tauM 10=\tau^1:\T M\ra M$, just  use the standard symbol $\tau$. The cotangent fibration is denoted by $\tau^\ast:\TT\ast M\ra M$.

Another important constructions are the \emph{iterated tangent bundles} $\Tt k M:=\T\hdots\T M$ and the \emph{iterated higher tangent bundles} $\TT {n_1}\hdots \TT {n_r} M$. Elements of the latter will be called \emph{ $(n_1,\hdots,n_r)$--velocities}.  For notational simplicity we shall also denote the iterated higher tangent functor $\T^{n_1} \T^{n_2}\ldots \T^{n_r}$ by $\TT{n_1, \ldots, n_r}$.

Of our particular interest will be the bundles $\TT k\TT l M=\TT{k,l}M$. These bundles admit natural projections to lower-order tangent bundles which will be denoted by $\tauM{(k,l)}{(k',l')}:\TT k\TT lM\ra\TT {k'}\TT {l'}M$ (for $k\geq k'$ and $l\geq l'$). (Iterated higher) tangent bundles are subject to a number of natural inclusions such as already mentioned $\iota^{1,k}:\TT {k+1} M\subset \T\TT kM$, $\iota^{l,k}:\TT {k+l}M\subset\TT l\TT k M$, $\iota^k:\TT k M\subset \Tt k M$, etc., which will be used  extensively.

 Occasionally, when dealing with manifolds other than $M$, or when it can lead to confusions, we will add a suffix to the maps $\iota^{\dots}_{\dots}$, $\tau^{\hdots}_{\hdots}$, etc., to emphasize which manifold we are working with, e.g., $\tauM kl=\tau^k_{l,M}$, etc.

Given a smooth function $f$ on a manifold $M$ one can construct functions $f^{(\alpha)}$ on $\TT k M$, with $0\leq \alpha\leq k$, the so called \emph{$(\alpha)$-lifts} of $f$ (see~\cite{Morimoto_Lifts}), defined by
\begin{equation}\label{eqn:alpha_lift}
f^{(\alpha)}(\jet k_0\gamma(t)):= \left.\frac{\dd^\alpha}{\dd t^\alpha}\right|_{t=0} f(\gamma(t)).
\end{equation}
The functions $f^{(k)}:\TT kM\to \R$ and $f^{(1)}:\T M\to \R$ are called the \emph{complete lift} and the \emph{tangent lift} of $f$, respectively.
By iterating this construction we obtain functions $f^{(\alpha, \beta)} := (f^{(\beta)})^{(\alpha)}$ on $\TT k \TT l M$ for $0\leq \alpha\leq k$, $0\leq \beta\leq l$, and, generally, functions $f^{(\eps_1, \ldots, \eps_r)}$ on $\TT {n_1} \ldots \TT {n_r} M$ for $0\leq \eps_j\leq n_j$, $1\leq j\leq r$.
A coordinate system $(x^a)$ on $M$ gives rise to the so-called \emph{adapted coordinate systems} $(x^{a, (\alpha)})_{0\leq \alpha \leq k}$ on $\TT k M$  and $(x^{a, (\eps)})_\eps$ on $\TT {n_1} \ldots \TT {n_r} M$ where the multi-index $\eps=(\eps_1, \ldots \eps_r)$ is as before, and $x^{a, (\alpha)}$, $x^{a, (\eps)}$ are obtained from $x^a$ by the above lifting procedure. Within this notation we easily find that the canonical inclusion $\iota^k:\TT kM\to \Tt k M$ is given by
\begin{equation}\label{eqn:coordinates_iota}
(\iota^k)^\ast (x^{a, (\eps)}) = x^{a,(\eps_1+\ldots+ \eps_k)},
\end{equation}
where $\eps\in \{0, 1\}^k$.

Let us remark that in the definition of $f^{(\alpha)}$ we follow the convention of \cite{Gay_Holm_Inn_inv_ho_var_probl_2012,Tulcz_diff_lagr_1975}, whereas  the original convention of \cite{Morimoto_Lifts} is slightly different, since it contains a normalizing factor $\frac 1{\alpha!}$ in front of the derivative.


\paragraph{Tangent lifts of vector bundles and canonical pairings.} Let us pass to another important tool for our analysis, namely the (iterated) higher tangent bundles of vector bundles. Let $\sigma:E\ra M$ be a vector bundle. It is clear that $\sigma$ may be lifted to  vector bundles $\TT k\sigma:\TT k E\ra\TT k M$, $\Tt k \sigma:\Tt k E\ra \Tt k M$, $\TT k\TT l\sigma:\TT k\TT l E\ra\TT k\TT l M$, etc.
Let $\sigma^\ast:E^\ast\ra M$ be the bundle dual to $\sigma$.

Throughout the paper we denote by $(x^a)$ the coordinates on the base $M$,
by $(y^i)$ the linear coordinates on the fibers of $\sigma$, and by $(\xi_i)$ the linear coordinates on the fibers of the dual bundle $\sigma^\ast:\E\ast\ra M$. Natural weighted coordinates on (iterated higher) tangent lifts of $\sigma$ and $\sigma^\ast$ are constructed from $x^a$, $y^i$, $\xi_j$ by the lifting procedure mentioned above. They are denoted by adding the proper degree to a coordinate symbol. Degrees will be denoted by bracketed small Greek letters: $(\eps)$, $(\alpha)$, $(\beta)$, etc.

The natural pairing $\<\cdot,\cdot>_{\sigma}$ between $E$ and $E^\ast$ induces a non-degenerated pairing
between  $\T E$ and $\T E^\ast$ over $\T M$:
$$
\<\cdot,\cdot>_{\T \sigma} := \<\cdot, \cdot>_{\sigma}^{(1)}: \T(E^\ast \times_M E) \simeq \T E^\ast\times_{\T M}\T E  \ra\R,
$$
i.e., $\<\cdot,\cdot>_{\T \sigma}$ is the tangent lift of $\<\cdot,\cdot>_{\sigma}$.
In a similar way, $\<\cdot,\cdot>_{\sigma}$ can be lifted  to non-degenerate pairings on higher tangent prolongations of $E$ and $E^\ast$.

\begin{prop}\label{prop:pairings} Let $\<\cdot,\cdot>_\sigma:E^\ast\times_ME\ra\R$ be the natural pairing. Let us define inductively
$$\<\cdot,\cdot>_{\Tt k\sigma}:\Tt kE^\ast\times_{\Tt k M}\Tt kE\ra\R$$
as the tangent lift of $\<\cdot,\cdot>_{\Tt {k-1}\sigma}$ (for $k\geq 1$) and let
$$\<\cdot,\cdot>_{\TT k\sigma}:\TT kE^\ast\times_{\TT k M}\TT kE\ra\R$$
be the restriction of $\<\cdot,\cdot>_{\Tt k\sigma}$ to the product of the subbundles $\TT k\sigma^\ast\subset\Tt k \sigma^\ast$
and $\TT k\sigma\subset\Tt k \sigma$. Then
\begin{enumerate}[(a)]
\item in the local coordinates denoted by
$(x^{a,(\epsilon)}, y^{i,(\epsilon)})$ and $(x^{a,(\epsilon)}, \xi_i^{(\epsilon)})$, with $\epsilon\in\{0,1\}^k$ (resp.
$(x^{a,(\alpha)}, y^{i,(\alpha)})$ and $(x^{a,(\alpha)}, \xi_i^{(\alpha)})$ with $0\leq \alpha \leq k$), on $\T^{(k)}E$ and $\T^{(k)}\E\ast$ (resp. $\T^{k}E$ and $\T^{k}\E\ast$),
\begin{equation}\label{eqn:pairingOne}
\<(x^{a,(\epsilon)}, \xi_i^{(\epsilon)}), (x^{a,(\epsilon)}, y^{i,(\epsilon)})>_{\T^{(k)}\sigma}=
\sum_{i}\sum_{\epsilon\in\{0,1\}^k} \xi_i^{(\epsilon)}\,y^{i,{(1, \ldots,1)-(\epsilon)}},
\end{equation}
\begin{equation}\label{eqn:pairingTwo}
\<(x^{a,(\alpha)}, \xi_i^{(\alpha)}), (x^{a,(\alpha)}, y^{i,(\alpha)})>_{\T^{k}\sigma}=
\sum_i\sum_{0\leq\alpha \leq k} \binom{k}{\alpha}\xi_i^{(\alpha)} \,y^{i,(k-\alpha)}.
\end{equation}
\item  $\<\cdot,\cdot>_{\Tt k\sigma}$ and $\<\cdot,\cdot>_{\TT k\sigma}$ are non-degenerate pairings.
\item $\<\cdot,\cdot>_{\TT k \sigma} = \<\cdot, \cdot>_{\sigma}^{(k)}$, i.e., the pairing $\<\cdot,\cdot>_{\TT k \sigma}$ is the complete lift of $\<\cdot, \cdot>_{\sigma}$ (up to the isomorphism $\TT k(E^\ast \times_M E) \simeq \TT k E^\ast\times_{\TT k M}\TT k E$).
\end{enumerate}
\end{prop}
\begin{proof}
We get (\ref{eqn:pairingOne}) by iterated differentiation of the function $\<\cdot, \cdot>_\sigma:
\left((x^a, \xi_i), (x^a, y^i)\right)\mapsto \sum_i \xi_i\,y^i$. Taking into account (\ref{eqn:coordinates_iota}) we find (\ref{eqn:pairingTwo}). Since the tangent lift of a non-degenerate pairing is non-degenerate, so it is for $\<\cdot,\cdot>_{\Tt k \sigma}$. The  non-degeneracy of the pairings $\<\cdot,\cdot>_{\Tt k \sigma}$ and $\<\cdot,\cdot>_{\TT k \sigma}$ can also be easily seen from the above local formulas. Finally, (c) follows immediately from the definition of the complete lift and the local expression \eqref{eqn:pairingTwo} of $\<\cdot,\cdot>_{\TT k \sigma}$.
\end{proof}
We stress that the above lifting procedure can be expressed in the framework of Weil functors and Frobenius algebras \cite{Weil_theo_point_proches_1953, Kolar_Michor_Slovak_nat_oper_diff_geom_1993}.


\paragraph{Canonical flip $\kappa_{k}$ and its dual $\veps_{k}$.}
It is well-known that the iterated tangent bundle $\T\T M$ admits an involutive double vector bundle isomorphism (called the \emph{canonical flip})
$$\kappa:\T\T M\lra\T\T M,$$
which intertwines the projections $\tau_{\T M}:\T\T M\ra\T M$ and $\T\tau:\T\T M\ra\T M$ (see, e.g.,  \cite{Mackenzie_lie_2005}). Such a notion of canonical flip can be generalized to a family of isomorphisms
$$\kappa_{k}:\TT k\T M\ra\T\TT kM,\quad \jet k_{t=0}\jet 1_{s=0}\gamma(t, s)\mapsto \jet 1_{s=0}\jet k_{t=0}\gamma(t, s),$$
which map the projection $\TT k\tau:\TT k\T M\ra\TT k M$ to $\tau_{\TT kM}:\T\TT kM\ra\TT kM$ over $\id_{\TT kM}$ and $\tau^k_{\T M}:\TT k\T M\ra \T M$ to $\T\tau^k:\T\TT kM\ra\T M$ over $\id_{\T M}$. Morphisms $\kappa_{k}$ can be also defined inductively as follows: $\kappa_{1}:=\kappa$ and $\kappa_{k+1}:=\T\kappa_{k}\circ\kappa_{\TT kM}\big|_{\TT{k+1}\T M}$, i.e., as the unique morphism making the diagram
\begin{equation}\label{eqn:kappa_kM}
\xymatrix{
\T\TT k\T M\ar[rr]^{\T\kappa_{k}} &&\T\T\TT kM\ar[rr]^{\kappa_{\TT kM}} && \T\T\TT kM\\
\TT{k+1}\T M\ar@{-->}[rrrr]^{\kappa_{k+1}}\ar@{_{(}->}[u] &&&&\T\TT{k+1}M\ar@{_{(}->}[u] }
\end{equation}
commutative. The local description of $\kappa_{k}$ is very simple. If $x^{a,(\alpha,\epsilon)}$ are natural coordinates on $\TT k\T M$ and $x^{a,(\epsilon,\alpha)}$ are natural adapted coordinates on $\T\TT k M$ (with $0\leq\alpha\leq k$ and $\epsilon=0,1$), then
$x^{a,(\alpha,\epsilon)}$  corresponds to  $x^{a,(\epsilon,\alpha)}$ via $\kappa_{k}$.

Introduce now the dual $\veps_{k}:\T^\ast\TT kM\ra\TT k\T^\ast M$ of the canonical flip $\kappa_{k}$ defined via the equality, 
\begin{equation}\label{eqn:kappa_eps}
\<\Psi,\kappa_{k}\circ V>_{\tau_{\TT kM}}=\<\veps_{k}\circ\Psi,V>_{\TT k\tau},
\end{equation}
where $V\in \TT k\T M$ and $\Psi\in\T^\ast\TT kM$ is a vector such that both pairings make sense (cf. \cite{Cantrin_Crampin_Inn_can_isom_1989}).
Formula \eqref{eqn:kappa_eps} shows that $\kappa_k$ and $\varepsilon_k$ are ``adjoint'' to each other with respect to canonical pairings, as schematically shown by the commutative diagram  
$$\xymatrix{&\T\TT k M\ar@{..>}[ld] &&\TT k\T M\ar@{..>}[rd]\ar[ll]_{\kappa_{k}}&\\
\R &\<\cdot,\cdot>_{\tau_{\TT kM}}&&\<\cdot,\cdot>_{\TT k\tau_M}&\R\\
&\T^\ast\TT k M \ar@{..>}[lu]\ar[rr]^{\veps_{k}} &&\TT k\T^\ast M.\ar@{..>}[ur]&}$$

In the coordinates $(x^{a, (\alpha)}, p_{a, (\alpha)} = \partial_{x^{a, (\alpha)}})$ on $\TT\ast\TT kM$ and $(x^{a, (\alpha)}, p_a^{(\alpha)})$ on $\TT k\TT\ast M$ (adapted from
standard coordinates $(x^{a, (\alpha)})$ on $\TT kM$, and $(x^a, p_a)$ on $\TT\ast M$, respectively), we find from \eqref{eqn:pairingTwo} that
\begin{equation}\label{eqn:eps_kM}
\veps_{k}\left(x^{a, (\alpha)}, p_{a, (\alpha)}\right) = \left(x^{a, (\alpha)}, p_a^{(\alpha)} = \binom{k}{\alpha}^{-1} p_{a, (k-\alpha)}\right).
\end{equation}


\newpage
\section{The main result}\label{sec:lemma}

In this section  the construction of the vector bundle morphisms $\Upsilon_{k,\sigma}$ and $\momenta_{k,\sigma}$, associated with an integer $k$ and a vector bundle $\sigma: E\ra M$, will be described. These morphisms are closely related with the geometric integration-by-parts procedure and will play a crucial role in the geometric construction of the Euler-Lagrange equations in the next Section \ref{sec:EL}.


\paragraph{Bundles of semi-holonomic vectors.}
\begin{df} 
For non-negative numbers $n_1, \ldots, n_r\geq 0$ let denote $\bar{n}:=n_1+\ldots+n_r$. The set
$$
\Tsemihol{n_1, \ldots, n_r} E=\left(\iota^{n_1,\ldots, n_r}_M\right)^\ast\T^{n_1, \ldots, n_r} E=
\{X\in\T^{n_1, \ldots, n_r} E: \T^{n_1, \ldots, n_r}\sigma(X)\in\T^{\bar{n}} M\subset\T^{n_1, \ldots, n_r}M\},
$$
consisting of all $(n_1,\ldots, n_r)$--velocities in $E$ projecting to $\bar{n}$--velocities (\emph{holonomic vectors}) in $\T^{n_1, \ldots, n_r} M$, is a vector subbundle of $\TT{n_1,\hdots,n_r}\sigma:\TT{n_1,\hdots,n_r}E\ra\TT{n_1,\hdots,n_r}M$. The restriction $\Tsemihol{n_1, \ldots, n_r} \sigma$ of $\TT{n_1,\hdots,n_r}\sigma$ to $\Tsemihol{n_1, \ldots, n_r} E$
 is \emph{the bundle of semi-holonomic vectors}. Given a vector bundle $\sigma':E'\to M'$ and a  morphism $\phi:\sigma\ra\sigma'$ we define $\Tsemihol{n_1,\hdots,n_r}\phi:\Tsemihol{n_1,\hdots,n_r}\sigma\ra\Tsemihol{n_1,\hdots,n_r}\sigma'$ as the restriction of $\TT{n_1,\hdots,n_r}\phi$ to $\Tsemihol{n_1,\hdots,n_r}E$. Thus $\Tsemihol{n_1,\hdots,n_r}$ is a functor in the category of vector bundles.

 In agreement with our previous notation we shall denote by $\Ttsemihol{k} E=\Tsemihol{1,\hdots,1}E$ the subbundle of semi-holonomic velocities in $\Tt k E=\TT{1,\hdots,1}E$.
\end{df}

In future considerations  the bundles $\Tsemihol{k,k}E$ and $\Ttsemihol{k} E$ will be of our special interest. Let us remark that although, in general, there is no canonical projection $\Tt k E\ra \TT k E$, there exists a natural projection  $\PP_k:\Ttsemihol k E\ra \TT k E$. It is defined as the left inverse to the canonical inclusion $\iota^k_E$ but considered as a map to $\Tsemihol{k} E$, as is explained in the following proposition.
\begin{prop}\label{prop:def_Pk} Consider a semi-holonomic vector $X\in\Ttsemihol{k}E \subset \Tt k E $ lying over the $k$--velocity $v^k\in\TT kM\subset\Tt kM$. Then the formula
$$
\<\PP_k(X),\Psi>_{\TT k\sigma}=\<X,\Psi>_{\Tt k\sigma},
$$
where $\Psi\in\TT k\E\ast\subset\Tt k\E\ast$ lies over $v^k$, defines a canonical projection $\PP_k: \Ttsemihol{k} E\to \TT kE$, i.e., $\PP_k\circ \iota_E^k = \id_{\TT k E}$.
Moreover, $P_k$ is a vector bundle morphism and it can be expressed in local coordinates as
$$\PP_k\left(x^{a,(\epsilon)},y^{i,(\epsilon)}\right)=\left(x^{a,(\alpha)},\overline{y}^{i,(\alpha)}\right),$$
where $ \overline{y}^{i,(\alpha)}=\binom k\alpha^{-1}\displaystyle \sum_{|\epsilon|=\alpha}y^{i,(\epsilon)}$ is the arithmetic average of all the coordinates of total degree $\alpha$.
\end{prop}
\begin{proof} It follows immediately from the properties of the non-degenerate pairing $\<\cdot, \cdot >_{\TT k \sigma}$ (see Proposition~\ref{prop:pairings}).
\end{proof}

\paragraph{The map $\bm{\Upsilon_{k,\sigma}}$.}

The result below describes the construction of a certain canonical and universal vector bundle morphism $\Upsilon_{k,\sigma}$ from $\Thol{k} \sigma$ to $\sigma$. The precise sense of the word \emph{universal} will be given later in Lemma \ref{lem:Upsilon_canonical}. Informally speaking, any other morphism  $\Thol{k} \sigma\ra\sigma$  can be derived in an easy way from $\Thol{k}$. In the next Section \ref{sec:EL} we show that this morphism is directly connected with integration by parts in the procedure of deriving the Euler-Lagrange equations.

To fix some notation denote an element
$\Phi\in \Thol{k} E\subset \TT k\TT k E$ by $\Phi^{(k,k)}$ and its projections to lower-order
velocities by
$$\Phi^{(m,n)}:=\tauE{(k,k)}{(m,n),E}(\Phi)\in\TT m\TT nE,$$
where $m,n\leq k$. Observe that since $\Phi$ lies over some $2k$--velocity, say $v^{2k}\in\TT {2k} M$, then all the elements $\Phi^{(m,n)}$ project under $\TT m\TT n\sigma$ to a fixed $m+n$--velocity $v^{m+n}\in \TT {m+n}M\subset \TT m\TT n M$ independently on the numbers in the sum $m+n$. In particular, different elements $\Phi^{(m,n)}$ with $m+n$ fixed, belonging \emph{a priori} to different bundles $\TT m\TT n E$, can be added together in the vector bundle $\Tt {m+n}\sigma:\Tt {m+n} E\ra\Tt {m+n} M$, which contains all of them. Denote by $\wt\Phi$ the following element of $\Tt k E$:
\begin{equation}\label{eqn:inner_Upsilon}
\wt\Phi:=
\Phi^{(0,k)}-\binom k 1\Phi^{(1,k-1)}+\binom k 2\Phi^{(2,k-2)}+\hdots+(-1)^k\Phi^{(k,0)}.
\end{equation}
 Similarly, define a morphism $\momenta_{k,\sigma}:\Tsemihol{k,k}\sigma \ra\TT k \sigma$ by the formula
\begin{equation}\label{eqn:moment_map}
\momenta_{k,\sigma}\left(\Phi^{(k,k)}\right):=
\PP_k\left[\binom{k+1}1\Phi^{(0,k)}-\binom{k+1}2\Phi^{(1,k-1)}+\hdots+
(-1)^{k+1}\binom{k+1}{k+1}\Phi^{(k,0)}\right],
\end{equation}
where $\PP_k:\Ttsemihol{k} E\ra\TT kE$ was defined in  Proposition~\ref{prop:def_Pk}.

Now we are ready to state the main result of this section.

\begin{thm}[Bundle-theoretic integration by part]\label{thm:k_hol} Let $\Phi=\Phi^{(k,k)}\in\Thol{k} E$ be a semi-holonomic vector projecting to $v^{2k}\in\TT{2k}M$ and let $v^k:=\tauM{2k}k(v^{2k})\in \TT k M$. Let
$\jet k\xi$ be any element in $\TT k_\xi E^\ast$ which projects to $v^k$ under $\TT k\sigma^\ast$.
Then, if $\wt\Phi$ is given by \eqref{eqn:inner_Upsilon}, the value of $\<\wt\Phi, \jet k\xi>_{\Tt k\sigma}$ does not depend on the choice of $\jet k \xi$. Hence it  defines a canonical vector bundle morphism $\Upsilon_{k,\sigma}:\Thol{k}\sigma\ra \sigma$ covering $\tau^{2k}:\TT {2k} M\ra M$ given by
\begin{equation}\label{eqn:Upsilon}
 \quad \<\Upsilon_{k,\sigma}(\Phi),\xi>_\sigma:=\<\wt\Phi, \jet k\xi>_{\Tt k\sigma} = \<\sum_{j=0}^k (-1)^j \binom k j \Phi^{(j,k-j)}, \jet k\xi>_{\Tt k\sigma}.
\end{equation}

Moreover,
\begin{enumerate}[(a)]
\item \label{cond:Upsilon_local} in coordinates, if $\Phi\sim\left(x^{a,(\alpha)},y^{j,(\beta,\gamma )}\right)$
where $0\leq \alpha\leq 2k$ and $0\leq \beta,\gamma\leq k$, then
\begin{equation}\label{eqn:Upsilon_local}
\Upsilon_{k, \sigma}(\Phi) = \left(x^a, \sum_{\alpha=0}^k (-1)^\alpha \binom{k}{\alpha} y^{i,(\alpha,k-\alpha)}\right);
\end{equation}

\item\label{cond:Upsilon_inductive} $\Upsilon_{k, \sigma}$ satisfies the recurrence formulas
\begin{align}
&\label{eqn:Upsilon_1}\Upsilon_{1,\sigma}(\Phi)=
\nu_E\left[\tau_{\T E}(\Phi)-\T\tau_E(\Phi)\right],\\
&\Upsilon_{k,\sigma}=\Upsilon_{1,\sigma}\circ\Upsilon_{k-1,\T\T \sigma}\Big|_{\Thol{k}E},\label{eqn:Upsilon_k}
\end{align}
where $\nu_E$ denotes the projection  $\T E\big|_M\cong E\times_M\T M \to E$;

\item \label{cond:momenta_inductive} $\momenta_{k, \sigma}$ satisfies the recurrence formulas
\begin{align}
&\momenta_{0,\sigma}=\id_{\sigma},\label{eqn:momenta_0}\\
&\<\momenta_{k,\sigma}\left(\Phi^{(k,k)}\right),\jet k\xi>_{\TT k\sigma}=\<\Upsilon_{k,\sigma}\left(\Phi^{(k,k)}\right),\xi>_{\sigma}+\<\momenta_{k-1,\T\sigma}\left(\Phi^{(k-1,k)}\right),\jet k\xi>_{\Tt k\sigma},\label{eqn:momenta_inductive}
\end{align}
where in the last term we consider $\Phi^{(k-1,k)}$ as a semi-holonomic vector in $\Tsemihol {k-1, k-1}\T E\supset \Tsemihol {k-1,k} E$;

\item\label{cond:int_by-parts} $\Upsilon_{k, \sigma}$ and $\momenta_{k-1,\sigma}$ are related by the ``bundle-theoretic integration by parts formula''
\begin{equation}\label{eqn:green}
\<\Phi^{(0,k)},\jet k\xi>_{\Tt k\sigma}=\<\Upsilon_{k,\sigma}\left(\Phi^{(k,k)}\right),\xi>_{\sigma}+\<\T\momenta_{k-1,\sigma}\left(\Phi^{(k,k-1)}\right),\jet k\xi>_{\Tt k\sigma},
\end{equation}
where in the last term we consider $\Phi^{(k,k-1)}$ as a vector in $\T\Tsemihol {k-1, k-1} E\supset \Tsemihol {k, k-1} E$;

\item\label{cond:T_Upsilon} morphism  $\Upsilon_{k,\cdot}$ commutes with the tangent functor $\T$ up to the canonical isomorphism $\wt\kappa_{k,k,E}:\Tsemihol{k,k} \T E\ra\T\Tsemihol{k,k} E$, i.e., the diagram
\begin{equation}\label{eqn:Upsilon_and_t}
\xymatrix{\Tsemihol{k,k}\T E\ar[rr]^{\Upsilon_{k,\T\sigma}}\ar[d]_{\wt\kappa_{k,k,E}}&& \T E\\
\T\Tsemihol{k,k}E\ar[urr]_{\T\Upsilon_{k,\sigma}}}
\end{equation}
is commutative;
\item\label{cond:T_momenta} morphism $\momenta_{k,\cdot}$ commutes with the tangent functor $\T$ up to the canonical isomorphisms $\wt\kappa_{k,k,E}:\Tsemihol{k,k} \T E\ra\T\Tsemihol{k,k} E$ and $\kappa_{k,E}:\TT k\T E\ra\T\TT kE$, i.e., the diagram
\begin{equation}\label{eqn:momenta_and_t}
\xymatrix{\Tsemihol{k,k}\T E\ar[rr]^{\momenta_{k,\T\sigma}}\ar[d]_{\wt\kappa_{k,k,E}}&& \TT k\T E\ar[d]^{\kappa_{k,E}}\\
\T\Tsemihol{k,k}E\ar[rr]_{\T\momenta_{k,\sigma}}&&\T\TT k E }
\end{equation}
is commutative.
\end{enumerate}
\end{thm}

Lemma \ref{lem:Upsilon_canonical} is a natural continuation of Theorem\ref{thm:k_hol}, but we decided to keep it separated since its proof is rather technical (see Appendix).
\begin{lem}[Universality of $\Upsilon_{k,\sigma}$]\label{lem:Upsilon_canonical} Let $M$ be a connected manifold. Then any functorial vector bundle morphism
 $(F_E, \und{F_E})$
\begin{equation}\label{e:FE}
\xymatrix{
\Thol{k} E \ar[d]_{\Thol k\sigma}\ar[rr]^{F_E}  && E \ar[d]_{\sigma} \\
\TT{2k} M \ar[rr]^{\und{F_E}}  && M
}
\end{equation}
is a linear combination of the morphisms
$\Upsilon_{l, \sigma}\circ \tau_{(l),E}^{(k)}$, where $\tau_{(l),E}^{(k)}: \Thol{k} E \to \Thol{l} E$ is the canonical projection induced by $\tau^{(k,k)}_{(l,l), E}: \TT {k,k} E \to \TT {l,l} E$ for $0\leq l\leq k$.
\end{lem}
\begin{proof} [Proof of Theorem~\ref{thm:k_hol}] We shall prove first that $\Upsilon_{k, \sigma}$ is a well-defined mapping and, simultaneously part \eqref{cond:Upsilon_inductive}. To this end, we proceed by induction with respect to $k$ for arbitrary $\sigma:E\ra M$.

Consider $k=1$. Take $\Phi=\Phi^{(1,1)}\in\T\T E$ and let $v^2:=\T\T\sigma(\Phi)\in\TT 2 M$. $\T\T \sigma$ has the structure of a double vector bundle with the projections $\T\tau_E$ and $\tau_{\T E}$ onto $\T E$. Vectors $\Phi^{(0,1)}=\tau_{\T E}(\Phi)$ and $\Phi^{(1,0)}=\T\tau_E(\Phi)$ project to the same vector $v^1=\tau_{\T M}(v^2)=\T\tau(v^2)$ in $\T M$ and to the same point $\tau_E(\Phi^{(0,1)})=\tau_E(\Phi^{(1,0)})=\Phi^{(0,0)}$ in $E$. It follows that their difference with respect to the vector bundle structure $\T\sigma:\T E\ra\T M$ belongs to $\T E\big|_M\cong E\times_M\T M$. Hence, from the properties of the pairing $\<\cdot,\cdot>_{\T\sigma}$,
$$\<\Phi^{(0,1)}-\Phi^{(1,0)},\jet 1\xi>_{\T\sigma}=\<\nu_E\(\Phi^{(0,1)}-\Phi^{(1,0)}),\xi>_\sigma.$$
In other words $\Upsilon_{1,\sigma}$ is well-defined and satisfies
\eqref{eqn:Upsilon_1}.

Assume now that the assertion holds for every $l\leq k-1$ and every vector bundle $\sigma$. Observe that, for every $l$,
$$\Thol{l} E\ni\Phi\longmapsto \sum_{i=0}^l(-1)^i\binom l i\Phi^{(i,l-i)}\in\Tt lE$$
is a functorial vector bundle morphism over $\tauM{2l}l:\TT {2l} M\lra\TT l M\subset\Tt lM$ being the combination, with constant coefficients, of functorial vector bundle morphisms $\Phi\mapsto\Phi^{(i,l-i)}$. In other words, given a vector bundle $\sigma':E'\ra M'$ and a vector bundle morphism $\alpha:\sigma\ra\sigma'$, it holds
$$\Tt l\alpha\left(\sum_{i=0}^l(-1)^i\binom l i\Phi^{(i, l-i)}\right)=
\sum_{i=0}^l(-1)^i\binom l i\left((\TT l\TT l\alpha)\Phi\right)^{(i,l-i)}.$$
This fact, combined with functoriality of $\<\cdot, \cdot>_\sigma$ and the inductive assumption, guarantees that
\begin{equation}\label{eqn:functoriality_Upsilon}
\alpha(\Upsilon_{l,\sigma}(\Phi))=\Upsilon_{l,\sigma'}((\TT l\TT l\alpha)\Phi)
\end{equation}
for every $l\leq k-1$.
In particular, we may take as $\alpha$ in \eqref{eqn:functoriality_Upsilon}, the following vector bundle morphisms $\T \tau_E:\T\T E\ra\T E$ over $\T\tau_M$, $\Psi\mapsto \Psi^{(1,0)}$  and $\tau_{\T E}:\T\T E\ra\T E$ over $\tau_{\T M}$, $\Psi\mapsto \Psi^{(0,1)}$ to get
\begin{align}
&\Big(\Upsilon_{k-1,\T\T \sigma}(\Phi)\Big)^{(1,0)}=\Upsilon_{k-1,\T \sigma}\left(\Phi^{(k,k-1)}\right) \label{eqn:01},\\
&\Big(\Upsilon_{k-1,\T\T \sigma}(\Phi)\Big)^{(0,1)}=\Upsilon_{k-1,\T \sigma}\left(\Phi^{(k-1,k)}\right) \label{eqn:10},
\end{align}
where we treat $\Phi^{(k,k)}$ as an element of $\TT {k-1, k-1}\TT {1,1} E$ while $\Phi^{(k,k-1)}$ and $\Phi^{(k-1,k)}$ as elements of $\TT {k-1, k-1}\T E$.
Now for the pairing \eqref{eqn:Upsilon} of our interest we can write:
\begin{align*}
&\<\Phi^{(0,k)}-\binom k 1\Phi^{(1,k-1)}+\binom k 2\Phi^{(2, k-2)}+\hdots+(-1)^k\Phi^{(k,0)}, \jet k\xi>_{\Tt k\sigma}= \\
&\qquad\phantom{-.}\<\Phi^{(0,k)}-\binom {k-1} 1\Phi^{(1,k-1)}+\binom {k-1} 2\Phi^{(2, k-2)}+\hdots+(-1)^{k-1}\Phi^{(k-1,1)}, \jet k\xi>_{\Tt k\sigma}+\\
&\qquad-\<\Phi^{(1,k-1)}-\binom {k-1} 1\Phi^{(2,k-2)}+\binom {k-1} 2\Phi^{(3,k-3)}+\hdots+(-1)^k\Phi^{(k,0)}, \jet k\xi>_{\Tt k\sigma}
\end{align*}
Thanks to our inductive assumption the right-hand side later equals
\begin{align*}
\<\Upsilon_{k-1,\T \sigma}\(\Phi^{(k-1,k)}),\jet 1\xi>_{\T\sigma}-\<\Upsilon_{k-1,\T \sigma}\(\Phi^{(k,k-1)}),\jet 1\xi>_{\T\sigma},
\end{align*}
 Now we can use equalities \eqref{eqn:01} and \eqref{eqn:10} to bring the above expression to the form
\begin{align*}
\<\Big(\Upsilon_{k-1,\T\T \sigma}(\Phi)\Big)^{(0,1)}-\Big(\Upsilon_{k-1,\T\T \sigma}(\Phi)\Big)^{(1,0)},\jet 1\xi>_{\T\sigma}\overset{\eqref{eqn:Upsilon_1}}=\big<\Upsilon_{1,\sigma}\big(\Upsilon_{k-1,\T\T\sigma}(\Phi)\big),\xi\big>_\sigma.
\end{align*}
This sums up to formula \eqref{eqn:Upsilon_k}
and assures that $\Upsilon_{k, \sigma}$ is indeed correctly defined.

The proof of \eqref{cond:Upsilon_local} is simple.
Locally, $\Phi$ and $\jet k \xi$ (in canonical induced graded coordinates on $\Thol{k}E$ and $\TT k\E\ast$) are given by
$$\Phi\sim\left(x^{a,(\alpha)},y^{j,(\beta,\beta' )}\right),\quad\text{and}\quad \jet k\xi\sim \left(x^{a,(\alpha)}, \xi_i^{(\beta)}\right);$$
where $\alpha=0,1,\hdots,2k$ and $\beta,\beta'=0,1,\hdots,k$.

We have shown that $\<\wt{\Phi}, \jet k \xi>_{\Tt k \sigma}$, which has a coordinate form of a polynomial in $\xi_i^{(\beta)}$ and $y^{j, (\beta', \beta'')}$, actually does not depend on $\xi_i^{(\beta)}$ for $\beta\geq 1$. Hence from \eqref{eqn:Upsilon} we get that $\Upsilon_{k,\sigma}(\Phi)=(x^a,\overline y^i)$, where $\overline y^i$ is a $\xi_i=\xi_i^{(0)}$-coefficient in the coordinate expression for $\<\wt\Phi,\jet k\xi>_{\Tt k\sigma}$. By \eqref{eqn:pairingOne}
we know that $\overline y^i$ is the coordinate of total degree $k$ in $\wt\Phi$. We conclude that the formula \eqref{eqn:Upsilon_local} for $\Upsilon_{k, \sigma}$ is true.

To prove \eqref{cond:momenta_inductive} we use the standard decomposition of binomial coefficients $\binom{k+1}{i+1}=\binom ki+\binom k{i+1}$ to obtain
\begin{align*}
&\binom{k+1}1\Phi^{(0,k)}-\binom{k+1}2\Phi^{(1,k-1)}+\hdots+(-1)^{k}\binom{k+1}{k+1}\Phi^{(k,0)}=\\
&\phantom{XXX}\left[\binom{k}0\Phi^{(0,k)}-\binom{k}1\Phi^{(1,k-1)}+\hdots+(-1)^{k}\binom{k}{k}\Phi^{(k,0)}\right]+\\
&\phantom{XXX}\left[\binom{k}1\Phi^{(0,1+k-1)}-\binom{k}2\Phi^{(1,1+k-2)}+\hdots+(-1)^{k-1}\binom{k}{k}\Phi^{(k-1,1+0)}\right].
\end{align*}
In the first expression on the right-hand side of this equality we easily recognize formula \eqref{eqn:Upsilon}, whereas, in the second, formula $\eqref{eqn:moment_map}_{k-1}$.

The proof of \eqref{cond:int_by-parts} is very similar. We decompose
\begin{align*}
\Phi^{(0,k)}=&\left[\binom{k}0\Phi^{(0,k)}-\binom{k}1\Phi^{(1,k-1)}+\hdots+(-1)^{k}\binom{k}{k}\Phi^{(k,0)}\right]+\\
&\left[\binom{k}1\Phi^{(1+0,k-1)}-\binom{k}2\Phi^{(1+1,k-2)}+\hdots+(-1)^{k}\binom{k}{k}\Phi^{(1+k-1,0)}\right].
\end{align*}
Now on the right-hand side of this equality we easily recognize formulas \eqref{eqn:Upsilon} and $\eqref{eqn:moment_map}_{k-1}$.

Finally, to prove \eqref{cond:T_Upsilon} and \eqref{cond:T_momenta}, observe that the tangent functor $\T$ commutes (up to canonical isomorphisms) with the projections
 $\tau^{k, k}_{i, k-i}$ and $\PP_k$. Thus it commutes (up to canonical isomorphisms) with $\Upsilon_{k,\cdot}$ and $\momenta_{k,\cdot}$ which are linear combinations of the compositions of the mentioned projections.
\end{proof}

\begin{rem}
Note that in formula \eqref{eqn:Upsilon_local} only the coordinates of the highest degree ($k$) matter but the coordinates of lower degrees may be non-trivial. For example, if $k=2$, we have $\wt\Phi=\Phi^{(0,2)}-2\Phi^{(1,1)}+\Phi^{(2,0)}\in\T\T E$ and $\Phi^{(0,2)}\in\TT 2E$, interpreted as an element of $\T\T E$ by means of $\iota^{1,1}:\TT 2E\ra\T\T E$, equals to $\left(x^{a,(0)},x^{a,(1)},x^{a,(1)},x^{a,(2)}, y^{i,(0,0)}, y^{i,(0,1)}, y^{i,(0,1)}, y^{i,(0,2)}\right)$. Similarly for $\Phi^{(2,0)}$. Therefore, in coordinates, $\wt\Phi$ looks like
$$\wt\Phi\sim\left(x^{a,(0)},x^{a,(1)},x^{a,(2)},\underline y^{i,(0,0)},\underline y^{i,(1,0)},\underline y^{i,(0,1)},\underline y^{i,(1,1)}\right),$$
where $\underline y^{i,(0,0)}=0$, $\underline y^{i,(1,0)}=-\underline y^{i,(0,1)}=y^{i,(1,0)}-y^{i,(0,1)}$ and $\underline y^{i,(1,1)}= y^{i,(2,0)}-2y^{i,(1,1)}+y^{i,(0,2)}$.
\end{rem}

\begin{rem}\label{rem:Upsilon_inductive}
Observe that, for any $k$ and $l$, the bundle $\Thol{l}\Thol{k}E$ is the pullback of $\TT {l,l}\TT {k,k}E$ with respect to the canonical inclusion $\TT {l,l}\iota^{k,k}\circ\iota^{l,l}_{\TT{2k}M}:\TT {2l, 2k} M\subset\TT {l,l} \TT {k, k} M$, i.e.,
\begin{align*}
\Thol{l}\left(\Thol{k}\sigma\right)=\Thol{l}\left(\left(\iota^{k,k}\right)^\ast\TT {k,k}\sigma\right)=\left(\iota^{l,l}_{\TT{2k}M}\right)^\ast\TT {l,l}\left(\left(\iota^{k,k}\right)^\ast\TT {k,k}\sigma\right)=\\
\left(\iota^{l,l}_{\TT{2k}M}\right)^\ast\left(\TT {l,l}\iota^{k,k}\right)^\ast\left(\TT {l,l}\TT {k,k}\sigma\right)=\left(\TT {l,l}\iota^{k,k}\circ\iota^{l,l}_{\TT{2k}M}\right)^\ast\left(\TT {l,l}\TT {k,k}\sigma\right).
\end{align*}
Consider now the inclusion $\TT{k+l, k+l}E\subset\TT {l,l}\TT {k,k}E$. Any semi-holonomic vector $X\in \Thol{k+l}E$ lying over $v^{2k+2l}\in\TT{2k+2l} M\subset\TT{k+l, k+l}M$ is mapped via this inclusion to an element lying over an element in $\TT{2l}\TT{2k}M\subset\TT {l, l}\TT {k,k}M$.  Thus we have the canonical inclusion $\Thol{k+l}E\subset \Thol{l}\Thol{k} E$. Therefore the inductive formula \eqref{eqn:Upsilon_k} can be described as follows:
$$\xymatrix{\Thol{k}E \ar@{^{(}->}[rr] \ar[d]^{\Upsilon_{k,\sigma}} &&\Thol{k-1}\Thol{1} E \ar[d]^{\Upsilon_{k-1,\Thol{1}\sigma}}\\
E && \Thol{1}E. \ar[ll]_{\Upsilon_{1,\sigma}}}$$
Expressing $\Upsilon_{k,\sigma}$ as a composition of morphisms $\Upsilon_{1,\cdot}$ according to \eqref{eqn:Upsilon_k} we can obtain the more general formula $\Upsilon_{k+l,\sigma}=\Upsilon_{k,\sigma}\circ\Upsilon_{l,\Tsemihol{k,k}\sigma}\big|_{\Tsemihol{k+l,k+l}E}$, i.e.,
\begin{equation}\label{eqn:Upsilon_inductive_general}
\xymatrix{\Thol{k+l}E \ar@{^{(}->}[rr] \ar[d]^{\Upsilon_{k+l,\sigma}} &&\Thol{l}\Thol{k} E \ar[d]^{\Upsilon_{l,\Thol{k}\sigma}}\\
E && \Thol{k}E. \ar[ll]_{\Upsilon_{k,\sigma}}}\end{equation}

Similarly, formula \eqref{eqn:momenta_inductive} generalizes to
\begin{equation}\label{eqn:momenta_inductive_general}
\begin{split}
&\<\momenta_{k+l,\sigma}\left(\Phi^{(k+l,k+l)}\right),\jet{k+l}\xi>_{\TT{k+l}\sigma}=\\
&\<\momenta_{k,\sigma}\left(\Upsilon_{l,\Tsemihol{k,k}\sigma}\left(\Phi^{(k+l,k+l)}\right)\right),\jet k\xi>_{\TT k\sigma}+
\< \momenta_{l-1,\TT{k+1}\sigma}\left(\Phi^{(l-1,l+k)}\right),\jet{k+l}\xi>_{\TT {k+l} \sigma}.\end{split}\end{equation}
The proof is left to the reader.

In light of \eqref{eqn:green}, we can interpret formulas \eqref{eqn:Upsilon_inductive_general} and \eqref{eqn:momenta_inductive_general} as follows: integration by parts $k+l$ times can be obtained as the composition of $k$ times and $l$ times integration by parts.
\end{rem}


\newpage
\section{Applications to variational calculus}\label{sec:EL}
In this section we use  Theorem~\ref{thm:k_hol} to give a geometric construction of the force and momentum in higher-order variational calculus (Theorem \ref{thm:var_calc}). As a consequence, in Theorem \ref{thm:EL}, we obtain necessary and sufficient conditions (Euler-Lagrange equations and transversality conditions) for a curve to be an extremal of the higher-order variational Problem \ref{prob:var_M}.

\paragraph{Formulation of the problem.}
Consider a \emph{Lagrangian function} $L:\TT k M\ra \R$ and the associated action
\begin{equation}\label{eqn:action_functional}
\gamma\longmapsto S_L(\jet k\gamma)=\int_{t_0}^{t_1}L(\jet k\gamma(t))\dd t,
\end{equation}
where $\gamma:[t_0,t_1]\ra M$ is a path and $\jet k\gamma(t)\in \TT k_{\gamma(t)} M$ its $\thh{k}$ prolongation. The set of  \emph{admissible paths} $\jet k\gamma$ will be denoted by $\Adm([t_0,t_1],\TT k M)$.

By an \emph{admissible variation} of an admissible path $\jet k\gamma(t)$ we will understand a curve $\del\jet k\gamma(t)\in\T_{\jet k\gamma(t)}\TT k M$. Observe that every admissible variation $\del\jet k\gamma(t)$ can be obtained from a homotopy $\chi(\cdot,\cdot):[t_0,t_1]\times(-\eps,\eps)\ra M$ such that $\chi(t,0)=\gamma(t)$ and $\del \jet k\gamma(t)=\jet 1_{s=0}\jet k\chi(t,s)$. We see that $\del \jet k\gamma$ is \emph{generated} by $\del\gamma(t)=\jet 1_{s=0}\chi(t,s)\in\T_{\gamma(t)} M$, i.e., a vector field along $\gamma(t)$, in the sense that
\begin{equation}\label{eqn:adm_var}
\del\jet k\gamma=\kappa_{k}\left(\jet k\del \gamma\right).
\end{equation}
So, $\del\gamma$ can be called a \emph{generator of the variation $\del\jet k\gamma$}.
Given an admissible variation $\del\jet k\gamma$ we may define the differential of the action $S_L$ in the direction of this variation:
$$\<\dd S_L(\jet k\gamma),\del\jet k\gamma>:=\int_{t_0}^{t_1}\<\dd L(\jet k\gamma(t)),\del\jet k\gamma(t)>_{\tau_{\TT k M}}\dd t.$$

Define now a natural projection
$$\Pp:\Adm([t_0,t_1],M)\ni\jet k\gamma\longmapsto \left(\jet {k-1}\gamma(t_0),\jet {k-1}\gamma(t_1)\right)\in \TT {k-1}M\times\TT {k-1}M,$$
which sends an admissible path $\jet k\gamma$ to the pair consisting of its initial and final $(k-1)$--velocity.  Its tangent map $\T\Pp$ sends a variation $\del\jet k\gamma\in\T_{\jet k\gamma}\TT k M$ to the pair $(\del\jet{k-1}\gamma(t_0),\del\jet{k-1}\gamma(t_1))\in\T\TT {k-1}M\times\T\TT{k-1} M$.

Now we are ready to formulate the following \emph{variational problem}.
\begin{problem}[Variational problem]\label{prob:var_M} For a given Lagrangian function $L:\TT kM\ra\R$ and a submanifold $S\subset \TT{k-1} M\times\TT{k-1} M$  which represents the admissible boundary values of $(k-1)$--velocities, find all curves $\gamma:[t_0,t_1]\ra M$ such that their $\thh{k}$ prolongation $\jet k\gamma$ satisfies
$$\<\dd S_L(\jet k\gamma),\del\jet k\gamma>=0\quad\text{for every admissible variation $\del\jet k\gamma$ such that $\T\Pp(\del \jet k\gamma)\in \T S$}.$$
\end{problem}
Let us comment the above formulation. In our approach to variational problems we study the behavior of the differential of the action functional in the directions of admissible variations (differential approach), rather to compare  the values of the action on nearby trajectories (integral approach). Hence, solutions of Problem \ref{prob:var_M} are only the critical, not extremal, points of the action functional \eqref{eqn:action_functional}. The philosophy of understanding a variational problem as the study of the differential of the action restricted to the sets of admissible trajectories and admissible variations allows one to treat the unconstrained and constrained cases in a unified way (see, e.g., \cite{KG_JG_var_calc_alg_2008, Gracia_Martin_Munos_2003, MJ_WR_nh_vac_2013}).

\paragraph{Higher-order variational calculus.}

Let $S_L$ be the action functional \eqref{eqn:action_functional} and $\del\jet k\gamma$ the variation \eqref{eqn:adm_var}. Define the \emph{force} $\F_{L,\gamma}(t)\in \TT\ast M$ and the \emph{momentum} $\M_{L,\gamma}(t)\in \TT{k-1}\TT\ast M$ along $\gamma(t)$ by
\begin{align}\label{eqn:force}
&\F_{L,\gamma}(t)=\Upsilon_{k,\tau^\ast}\left(\jet k\Lambda_L(\jet k\gamma(t))\right),\\\label{eqn:momentum}
&\M_{L,\gamma}(t)=\momenta_{k-1,\tau^\ast}\left(\jet{k-1}\lambda_L(\jet k\gamma(t))\right),
\end{align}
 where $\Lambda_L:=\veps_{k}\circ\dd L:\TT kM\ra\TT k\TT\ast M $ and $\lambda_L:=\tauM k{k-1,\TT\ast M}\circ \Lambda_L:\TT kM\ra\TT{k-1}\TT\ast M$.
\begin{thm}\label{thm:var_calc} The differential of the action $S_L$ in the direction of the variation $\del\jet k\gamma$ equals
\begin{equation}\label{eqn:action_var_full}
\<\dd S_L(\jet k\gamma),\del\jet k\gamma>=
\int_{t_0}^{t_1}\<\F_{L,\gamma}(t),\del\gamma(t)>_{\tau}\dd t+\<\M_{L,\gamma}(t),\jet {k-1}\del \gamma(t)>_{\TT {k-1} \tau}\Bigg|_{t_0}^{t_1}.
\end{equation}
\end{thm}

Theorem~\ref{thm:EL} below is  an immediate consequence of formula \eqref{eqn:action_var_full}.

\begin{thm}\label{thm:EL}
A curve $\gamma$ is a solution of Problem \ref{prob:var_M} if and only if it satisfies the following \emph{Euler-Lagrange (EL) equation}
\begin{align}\label{eqn:EL}
&\F_{L, \gamma}(t)=0
\intertext{and the \emph{transversality conditions}}\label{eqn:transv}
&\left(-\veps_{k-1}^{-1}\left(\M_{L,\gamma}(t_0)\right),\veps_{k-1}^{-1}\left(\M_{L,\gamma}(t_1)\right)\right)\in\TT\ast\TT{k-1}M\times \TT\ast\TT{k-1}M\quad \text{annihilates $\T S$}.
\end{align}
\end{thm}


\begin{proof}[Proof of Theorems \ref{thm:var_calc} and \ref{thm:EL}]
Let us calculate the variation of the action $S_L$ in the direction $\del\jet k\gamma$:
\begin{align*}
&\<\dd L(\jet k\gamma(t)),\del\jet k\gamma(t)>_{\tau_{\TT kM}}\overset{\eqref{eqn:adm_var}}=\<\dd L(\jet k\gamma(t)),\kappa_{k}\left(\jet k\del\gamma(t)\right)>_{\tau_{\TT kM}}\overset{\eqref{eqn:kappa_eps}}=\\
&\<\veps_{k}\circ\dd L(\jet k\gamma(t)),\jet k\del\gamma(t)>_{\TT k\tau}=\<\Lambda_L(\jet k\gamma(t)),\jet k\del \gamma(t)>_{\TT k\tau}.
\end{align*}

Now we can use formula \eqref{eqn:green} with $\Phi^{(k,k)}=\jet k \Lambda_L(\jet k\gamma(t))$, $\Phi^{(0,k)}=\Lambda_L(\jet k\gamma(t))$ and $\Phi^{(k,k-1)}=\jet k \lambda_L(\jet k\gamma(t))$, since the element $\Phi^{(k,k)}:=\jet k\Lambda_L(\jet k\gamma(t))\in\TT k\TT k\TT \ast M$ is a semi-holonomic vector as it projects to $\jet k\jet k\gamma=\jet{2k}\gamma\in\TT{2k}M$. We get
\begin{align*}
&\<\Lambda_L(\jet k\gamma(t)),\jet k\del \gamma(t)>_{\TT k\tau}\overset{\eqref{eqn:green}}{=}\\
&\<\Upsilon_{k,\tau^\ast}\left(\jet k\Lambda_L(\jet k\gamma(t))\right),\del\gamma(t)>_\tau+\<\T\momenta_{k-1,\tau^\ast}\left(\jet k\lambda_L(\jet k\gamma(t))\right),\jet k\del\gamma(t)>_{\Tt k\tau}.
\end{align*}
In the first summand we recognize the force $\F_{L,\gamma}(t)$ defined by \eqref{eqn:force}. To the second we can apply the equality
$$
\T\momenta_{k-1,\tau^\ast}\left(\jet k\lambda_L(\jet k\gamma(t))\right)=
\T\momenta_{k-1,\tau^\ast}\left(\jet 1 \jet {k-1}\lambda_L(\jet k\gamma(t))\right)=
\jet 1\left[\momenta_{k-1,\tau^\ast}\left(\jet {k-1}\lambda_L(\jet k\gamma(t))\right)\right].$$ We conclude that
\begin{align*}
&\<\Lambda_L(\jet k\gamma(t)),\jet k\del \gamma(t)>_{\TT k\tau}=\\
&\<\F_{L,\gamma}(t),\del\gamma(t)>_\tau+\<\jet 1\left[\momenta_{k-1,\tau^\ast}\left(\jet{k-1}\lambda_L(\jet k\gamma(t))\right)\right],\jet 1\jet {k-1}\del\gamma(t)>_{\T\TT {k-1}\tau}=
\\
&\<\F_{L,\gamma}(t),\del\gamma(t)>_\tau+\frac{\dd}{\dd t}\<\momenta_{k-1,\tau^\ast}\left(\jet{k-1}\lambda_L(\jet k\gamma(t))\right),\jet {k-1}\del\gamma(t)>_{\TT {k-1}\tau}=
\\
&\<\F_{L,\gamma}(t),\del\gamma(t)>_\tau+\frac{\dd}{\dd t}\<\M_{L,\gamma}(t),\jet {k-1}\del\gamma(t)>_{\TT{k-1}\tau},
\end{align*}
where the momentum $\M_{L,\gamma}(t)$ is defined by \eqref{eqn:momentum}. Thus the variation \eqref{eqn:action_var_full} reads
\begin{equation}\label{eqn:action_infinitesimal}
\<\dd L(\jet k\gamma(t)),\del\jet k\gamma(t)>_{\tau_{\TT kM}}=\<\F_{L,\gamma}(t),\del\gamma(t)>_\tau+\frac{\dd}{\dd t}\<\M_{L,\gamma}(t),\jet {k-1}\del\gamma(t)>_{\TT{k-1}\tau}.
\end{equation}
Integrating the above expression over $[t_0,t_1]$ we get \eqref{eqn:action_var_full}, concluding the proof of Theorem \ref{thm:var_calc}. Theorem \ref{thm:EL} follows easily, as $\del\jet{k-1}\gamma(t)=\kappa_{k-1}(\jet{k-1}\del\gamma(t))$ and $\veps_{k-1}$ is dual to $\kappa_{k-1}$, in light of equation \eqref{eqn:kappa_eps}.
\end{proof}


\begin{rem}\label{rem:EL_full}
The process of constructing  the EL equations \eqref{eqn:EL} starting from the Lagrangian function $L$ can be followed on the diagram
\begin{equation}\label{eqn:EL_diagram}
\xymatrix{
&&&&&& \ker\Upsilon_{k,\tau^\ast}\ar@{^{(}->}[d]\\
&&\TT\ast\TT kM \ar[rr]^{\eps_{k}}\ar[d]^{\left(\tau_{\TT k M}\right)^\ast} && \TT k\TT\ast M\ar[d]^{\TT k\tau^\ast}\ar@{-->}[rr]^{\jet k\Lambda_L(\jet k\gamma)}\ar@{-->}[rru]^{\exists ?}&&\Thol{k}\TT\ast M\ar[d]^{\Upsilon_{k,\tau^\ast}}\\
I\ar[rr]^{\jet k\gamma}&&\TT kM\ar@{=}[rr] \ar@/^0.5pc/@{..>}[u]^{\dd L}\ar@{..>}[rru]_{\Lambda_L} && \TT k M && \TT\ast M.
}\end{equation}

Similarly, the geometric construction of the momenta \eqref{eqn:momentum} corresponds to the diagram
\begin{equation}\label{eqn:momenta_diagram}
\xymatrix{
&&\TT\ast\TT kM \ar[rr]^{\eps_{k}}\ar[d]^{\left(\tau_{\TT kM }\right)^\ast} && \TT k\TT\ast M \ar[d]^{\tau^{k}_{k-1,\TT\ast M}}&&&&\\
I\ar[rr]^{\jet k\gamma}&&\TT kM\ar@{..>}[rr]^{\lambda_L} \ar@/^0.5pc/@{..>}[u]^{\dd L}\ar@{..>}[rru]^{\Lambda_L} && \TT{k-1}\TT\ast M\ar@{-->}[rr]^{\jet{k-1}\lambda_L(\jet k\gamma)} && \Thol{k-1}\TT\ast M\ar[rr]^{\momenta_{k-1,\tau^\ast}}&& \TT{k-1}\TT\ast M.
}\end{equation}
Above we used dotted arrows to denote the objects associated with the Lagrangian function,  whereas dashed arrows are maps defined only along the images of $\gamma(t)$.

Note also that the map $\Upsilon_{k,\tau^\ast}$ allows us to define \emph{higher-order EL equations with external forces}. Namely, given an \emph{external force}, i.e., a map $F:[t_0,t_1]\ra\TT\ast M$, we can consider equation
$$\Upsilon_{k,\tau^\ast}\left(\jet k\Lambda_L(\jet k\gamma(t))\right)=F(t).$$
When $F(t)=0$, the equation above reduces to the EL equation \eqref{eqn:EL}.
\end{rem}

\paragraph{Local form of the forces and momenta.}
We shall now derive the local form of the force \eqref{eqn:force} and momentum \eqref{eqn:momentum}.

Let us first calculate the force. Consider a trajectory $\gamma(t)\sim (x^a(t))$. The differential $\dd L(t^k\gamma(t)) \in \TT\ast\TT kM$ is
given by $p_{a, (\alpha)} = \frac{\partial L}{\partial x^{a,(\alpha)}}(t^k\gamma(t))$, hence using \eqref{eqn:eps_kM} we calculate the coordinates of
$\Lambda_L(\jet k\gamma(t)) \in \TT k\TT\ast M$, namely,
$$
p_a^{(\alpha)}(\Lambda_L(\jet k\gamma(t))) =  \binom{k}{\alpha}^{-1}\frac{\partial L}{\partial x^{a,(k-\alpha)}}(\jet k\gamma(t)), \quad x^{a, (\alpha)}(\Lambda_L(\jet k\gamma(t))) = \frac{\dd^\alpha x^a(t)}{\dd t^\alpha}.
$$
Therefore, $p_a^{(\alpha, \beta)}(\jet k\Lambda_L(\jet k\gamma(t))) = \binom{k}{\beta}^{-1}\frac{\dd^\alpha}{\dd t^\alpha}\frac{\partial L}{\partial x^{a,(k-\beta)}}$, and using
\eqref{eqn:Upsilon_local} we find that our formula \eqref{eqn:force} of the force takes the following  well-known form
\begin{equation}\label{eqn:force_local}
\F_{L,\gamma}(t)_a=\sum_{\alpha=0}^k (-1)^\alpha \frac{\dd^\alpha}{\dd t^\alpha}\left(\frac{\partial L}{\partial x^{a,(\alpha)}}(t^k\gamma(t))\right).
\end{equation}

 Concerning momentum, a direct (i.e., by means of formulas \eqref{eqn:moment_map} and \eqref{eqn:momenta_diagram}) derivation of its local form  is a quite complicated computational task, so that we will obtain it by using formula \eqref{eqn:action_infinitesimal}, instead.

To this end, consider  a generator of the variation $\del\gamma(t)\sim(x^a(t),\del x^a(t))$ and its $\thh{k}$ tangent lift $\jet k\gamma(t)\sim\left(x^{a,(\alpha)}(t),\del x^{a,(\alpha)}(t)\right)_{\alpha=0,\hdots,k}$, where $x^{a,(0)}(t)=x^a(t)$, $\del x^{a,(0)}(t)=\del x^a(t)$ and $x^{a,(\alpha+1)}(t)=\frac{\dd}{\dd t}x^{a,(\alpha)}(t)$, $\del x^{a,(\alpha+1)}(t)=\frac{\dd}{\dd t}\del x^{a,(\alpha)}(t)$, for $\alpha=0,\hdots,k-1$. Let the momentum $\M_{L,\gamma}(t)$ along the trajectory $\gamma(t)$ be locally given by $\left(x^{a,(\alpha)}(t),\binom {k-1}\alpha ^{-1} p_{a,(k-1-\alpha)}(t)\right)$. Thus, by $\eqref{eqn:eps_kM}_{k-1}$, $\eps^{-1}_{k-1}\left(\M_{L,\gamma}(t)\right)\sim\left(x^{a,(\alpha)}(t), p_{a,(\alpha)}(t)\right)\in\TT\ast\TT kM$. Locally
$$
\<M_{L,\gamma}(t), \jet{k-1}\del\gamma(t)>_{\TT{k-1}\tau}=\<\eps^{-1}_{k-1}\left(M_{L,\gamma}(t)\right), \del\jet{k-1}\gamma(t)>_{\tau_{\TT{k-1}M}}=\sum_a\sum_{\alpha=0}^{k-1} p_{a,(\alpha)}(t)\del x^{a,(\alpha)}(t).
$$
Thus
$$
\frac{\dd}{\dd t}\<M_{L,\gamma}(t), \jet{k-1}\del\gamma(t)>_{\TT{k-1}\tau}=\sum_a\sum_{\alpha=0}^{k} \left(p_{a,(\alpha-1)}(t)+\dot p_{a,(\alpha)}(t)\right)\del x^{a,(\alpha)}(t),
$$
where in the last formula we take $p_{a,(k)}(t)=p_{a,(-1)}(t)=0$. The left-hand side of  \eqref{eqn:action_infinitesimal} equals

\noindent $\sum_a\sum_{\alpha=0}^k \frac{\pa L}{\pa x^{a,(\alpha)}}(\jet k\gamma(t))\del x^{a,(\alpha)}$, so from \eqref{eqn:force_local} we get
\begin{align*}
\sum_a\sum_{\alpha=0}^k\del x^{a,(\alpha)}(t)\left(p_{a,(\alpha-1)}(t)+\dot p_{a,(\alpha)}(t) \right)= \\
\sum_a\sum_{\alpha=0}^k \frac{\pa L}{\pa x^{a,(\alpha)}}(\jet k\gamma(t))\del x^{a,(\alpha)}(t)- \sum_a F_{L,\gamma}(t)_a \del x^{a,(0)}(t).
\end{align*}
Since at a fixed time $t$ the variation $\del x^{a,(\alpha)}(t)$ can be arbitrary, above equation splits into the following set of linear equations
\begin{align*}
&p_{a,(k-1)}(t)=\frac{\pa L}{\pa x^{a,(k)}}(\jet k\gamma(t)),\\
&\dot p_{a,(k-1)}(t)+p_{a,(k-2)}(t)=\frac{\pa L}{\pa x^{a,(k-1)}}(\jet k\gamma(t)),\\
&\hdots\\
&\dot p_{a,(1)}(t)+p_{a,(0)}(t)=\frac{\pa L}{\pa x^{a,(1)}}(\jet k\gamma(t)),\\
&\dot p_{a,(0)}(t)=\frac{\pa L}{\pa x^{a,(0)}}(\jet k\gamma(t))-F_{L,\gamma}(t)_a,
\end{align*}
whose unique solution is
\begin{equation}\label{eqn:momentum_local}
p_{a,(\alpha)}(t)=\sum_{\beta=0}^{k-1-\alpha} (-1)^\beta\frac{\dd^\beta}{\dd t^\beta }\left(\frac{\pa L}{\pa x^{a,(\alpha+\beta+1)}}(\jet k\gamma(t))\right),\quad\text{for $\alpha=0,\hdots,k-1$;}
\end{equation}
i.e., the well-known formula for momenta.


\newpage
\section{Examples and perspectives}\label{sec:examples}

\paragraph{Tulczyjew's approach to higher-order geometric mechanics.}
The problem of geometric formulation of higher-order variational calculus on a manifold $M$ has a few solutions. The first approach is due to Tulczyjew \cite{Tulcz_diff_lagr_1975,Tulcz_lagr_diff_1976}, who gave a geometric construction of a 0-order derivation $\mathcal{E}:\Sec(\TT\ast \TT\infty M)\ra\Sec(\TT\ast\TT\infty M)$ such that $\mathcal{E}(\dd L)=0$ are the higher-order Euler-Lagrange equations for any Lagrangian function $L:\TT k M\ra \R$. Tulczyjew expressed his construction using the language of derivations and  infinite jets. The latter allowed him to cover all orders $k$ by a unique operator. Later Tulczyjew's theory was interestingly extended by Crampin, Sarlet and Cantrijn \cite{Crampin_Sar_Cart_high_ord_diff_eqns_1986}.

Another approach was inspired by Tulczyjew's papers \cite{Tulcz_dyn_ham_1976,Tulcz_dyn_lagr_1976} on the first order mechanics. The idea was to generate the EL equations from a Lagrangian submanifold. Two similar solutions where given by  Crampin \cite{Crampin_EL_higher_order_1990} and de Leon and Lacomba \cite{Leon_Lacomb_lagr_sbmfd_ho_mech_sys_1989}. They constructed the equations from a Lagrangian  submanifold in $\T\TT{k-1}\TT\ast M$ or $\T\TT\ast\TT{k-1} M$ generated by the Lagrangian $L$ .

All these solutions have, however, some drawbacks. First of all, they describe only a part of the Lagrangian formalism, namely the EL equations, whereas the full structure of variational calculus should contain also momenta (boundary terms). Secondly, the correctness of these constructions is checked in coordinates. Note, however, that it is not the coordinate expression that defines the EL equations, but the opposite: we deduce the right local expression from the proper variational principle. Therefore a fully satisfactory geometric construction should somehow explain the steps performed while deriving the known form of the equations (as it is in our approach described in the previous Section \ref{sec:EL}), not give a black-box answer.

In later years Tulczyjew \cite{Tulcz_ehres_jet_theory_2006} extended his work to give a full description of higher-order Lagrangian formalism (i.e., including momenta) in the language of derivations. Another approach was communicated to us by Grabowska \cite{KG_private}, who derived the $\thh{k}$ order formalism as the first order formalism on $\TT{k-1}M$. Her ideas are, to some extend, similar to these of \cite{Crampin_EL_higher_order_1990, Leon_Lacomb_lagr_sbmfd_ho_mech_sys_1989}, where the canonical inclusion $\TT kM\subset \T\TT{k-1} M$ was also used. Another quite general approach to the topic was presented by A.M. Vinogradov and his collaborators in the framework of secondary calculus (see Section 3 of \cite{Viatgliano_lagr_ham_h_field_theor_2010} and the references therein).

Below we relate Tulczyjew's approach to our results from the previous Section \ref{sec:EL}.

\paragraph{Comparison with Tulczyjew's formulas.}
In \cite{Tulcz_lagr_diff_1976,Tulcz_ehres_jet_theory_2006} Tulczyjew introduced graded derivations of degree 0
\begin{align*}
&\mathcal{E}:=\sum_{n=0}^k\frac{(-1)^n}{n!}\left(\tauM{2k}{k+n}\right)^\ast(\dd_\T)^n\iota_{F_n}:\OmegaALL\left(\TT kM\right)\lra\OmegaALL\left(\TT {2k}M\right)
\intertext{and}
&\mathcal{P}:=\sum_{n=1}^k\frac{(-1)^n}{n!}\left(\tauM{2k-1}{k+n-1}\right)^\ast(\dd_\T)^{n-1}\iota_{F_n}:\OmegaALL\left(\TT kM\right)\lra\OmegaALL\left(\TT {2k-1}M\right)
\end{align*}
defined by means of two basic derivations, namely:
\begin{enumerate}[(i)]
\item the total derivative $\dd_\T:\OmegaALL\left(\TT sM\right)\ra\OmegaALL\left(\TT{s+1}M\right)$ being a $d^\ast$--derivation characterized by $\dd_\T f^{(\alpha)} = f^{(\alpha+1)}$ for $0\leq \alpha\leq s$, where $f^{(\alpha)}$ denotes the $(\alpha)$--lift of a smooth function $f$ on $M$ as defined in \eqref{eqn:alpha_lift}, and
\item the $i^\ast$--derivations $\iota_{F_n}:\OmegaALL(\TT kM)\ra\OmegaALL(\TT kM)$, associated with the canonical (nilpotent) endomorphism $F_n: \T \TT kM\ra \T \TT kM$ of the tangent bundle,
    $$F_n\left(\jet 1_{s=0} \jet k_{t=0}\gamma(s,t)\right):=\jet 1_{s=0}\left(\jet k_{t=0}\gamma(st^n,t)\right).$$
\end{enumerate}
Recall that a derivation $a$ of the algebra of differential forms $\OmegaALL(M)$ is called a $d^\ast$-derivation (resp. $i^\ast$-derivation) if $a$ commutes with de Rham differential (resp. if $a$ vanishes on $\Omega^0(M)=\mathcal{C}^\infty(M)$) (\cite{Tulcz_lagr_diff_1976}, see also  \cite{Kolar_Michor_Slovak_nat_oper_diff_geom_1993}, chapter $8$).
As the algebra $\OmegaALL(M)$ is generated by $\Omega^0(M)$ and $\Omega^1(M)$, hence  a $d^\ast$-derivation (resp. $i^\ast$-derivation) is fully determined by its values on $\Omega^0(M)$ (resp. $\Omega^1(M)$).  The values of $d_T$ on $\Omega^0(\TT k M)$ are given in the above characterization of $d_T$, while for $\iota_{F_n}$ one defines $\<\iota_{F_n}\mu, v> := \<\mu, F_n v>$
    for a $1$-form $\mu$ and a tangent vector $v$ on $\TT k M$. Note that $F_n = F_1^n$ and $F_1$ is the \emph{canonical higher almost tangent structure} on $\TT kM$ \cite{Leon_Rodrigues_higher_almost_tangent}.

It turns out that the form $\mathcal{E}(\dd L)\in\Omega^1(\TT{2k}M)$ is vertical with respect to the projection $\tau^{2k}:\TT{2k}M\ra M$ and that  the form $\mathcal{P}(\dd L)\in\Omega^1(\TT{2k-1}M)$ is vertical with respect to $\tauM{2k-1}{k-1}:\TT{2k-1}M\ra\TT{k-1} M$. Therefore taking the appropriate vertical parts of these forms we can define operators
$$\mathcal{EL}:\TT{2k}M\lra\TT\ast M\quad\text{and}\quad\mathcal{PL}:\TT{2k-1}M\lra\TT\ast\TT{k-1}M.$$
Formulas $\mathcal{EL}(\jet{2k}\gamma)=\Upsilon_{k,\tau^\ast}\left(\jet k\Lambda_L(\jet k\gamma)\right)$
and
$\mathcal{PL}(\jet{2k-1}\gamma)=\veps_{k-1}\left(\momenta_{k-1,\tau^\ast}\left(\jet{k-1}\lambda_L(\jet k\gamma)\right)\right)$
relate Tulczyjew's constructions to ours.

\paragraph{Applications to mechanics on algebroids.}
In \cite{MJ_MR_higher_algebroids_2013} we showed that with every almost-Lie algebroid structure on the bundle $\sigma:E\ra M$, one can canonically
associate an infinite tower of graded bundles
\begin{equation}\label{eqn:tower_projection}
\hdots \lra E^k\lra E^{k-1}\lra\hdots\lra E^1=E
\end{equation}
 equipped with a family of graded-bundle relations
$$\kappa_k:\TT kE\relto\T E^k.$$
The relation $\kappa_k$ is of special kind -- it is dual of a vector bundle morphism $\veps_k:\T^\ast E^k \ra \TT k E^\ast$.
A natural example of such a structure is provided by the higher tangent bundles $E^k=\TT kM$ together with the canonical flips $\kappa_k:\TT k\T M\ra\T\TT kM$. Another example is $E^k:=\TT k_e G$, the higher tangent space at identity $e\in G$ to a Lie group $G$. Both examples should be considered as the extreme cases of what should we call a \emph{higher algebroid}. Except for $k=1$ there is no Lie bracket on sections of $E^k$. It is the relation $\kappa_k$ which is responsible for the algebraic structure on $E^k$.
More general examples can be obtained by reducing higher tangent bundles of Lie groupoids.

Given a smooth function $L:E^k\ra\R$ one can naturally define a variational problem on $E^k$. Such problems cover, on one hand, the standard variational problems like Problem \ref{prob:var_M} (in which case $E^k=\TT kM$), and, on the other hand, the reduction of  invariant higher-order variational problems on a Lie groupoid. We showed in \cite{MJ_MR_higher_algebroids_2013} that for such problems an analog of Theorem \ref{thm:var_calc} holds, as well. Thus, we can characterize the variation of an action by means of the \emph{force} $\Upsilon_{k,\sigma^\ast}\left(\jet k\veps_k \left(\dd L(\jet k\gamma(t))\right)\right)$ and \emph{momentum} $\momenta_{k-1,\sigma^\ast}\left(\jet{k-1}\tau^k_{k-1}\left(\veps_k \left(\dd L(\jet k\gamma(t))\right)\right)\right)$, where now $\veps_k:\TT\ast E^k\ra\TT k E^\ast$ is the dual of the relation $\kappa_k$, $\tau^k_{k-1}:E^k\ra E^{k-1}$ is the tower projection \eqref{eqn:tower_projection}, $\sigma^\ast:E^\ast\ra M$ is the dual of $\sigma$, while $\Upsilon_{k,\sigma^\ast}$ and $\momenta_{k,\sigma^\ast}$ are the same maps introduced in Section \ref{sec:lemma}.


\paragraph{Example: Riemannian cubic polynomials.}
Let us consider one of the simplest, but interesting,  second-order variational problem: given an integer $n\geq 2$ and
points $a=x_1<x_2<\ldots<x_n=b$ on the real line $\R$, and values $y_1, y_2, \ldots, y_n, v_a, v_b\in\R$,
find an $f\in C^2([a,b])$ such that $f(x_i)=y_i$ for $1\leq i\leq n$ and $f'(a)=v_a$, $f'(b)=v_b$ which
minimizes the integral $\int_a^b f''(x) dx$.  This problem has a unique solution, called a \emph{complete cubic spline} \cite{Boor_splines_1978}, which is a piece-wise cubic polynomial $P$ determined uniquely
by the following properties: it is a polynomial of degree $\leq 3$ on each interval $[x_i, x_{i+1}]$, $1\leq i\leq n-1$,
$P'(a) =v_a$, $P'(b) =v_b$, and it has continuous second derivatives at each ``slope'' $x_2, \ldots, x_{n-1}$.
Note that the EL equations \eqref{eqn:force_local} read as $f^{(4)} = 0$, i.e., $f$ is locally a polynomial of degree $\leq 3$.

Above example  generalizes to a variational problem on any Riemannian manifold $(M, g)$.
Indeed, let $\nabla$ denote the Levi-Civita connection for the metric $g$ and, for a smooth curve $\gamma: \R \to M$
denote by $D_t$ the covariant derivative $\nabla_{\dot{\gamma}(t)}$ along  $\gamma$.
Following \cite{NHP_splines_surfaces_1989} we define
\begin{equation}\label{e:Riemanniancubic}
L(\jet 2_{t=0}\gamma(t)) := g_{\gamma(0)}(D_t|_{t=0}\dot{\gamma}(t),  D_t|_{t=0}\dot{\gamma}(t)).
\end{equation}
Locally, for $\gamma(t)\sim(x^a(t))$,
$$
D_t|_{t=0}\dot{\gamma}(t) = (\ddot{x}^c + \Gamma_{ab}^c(x) \dot{x}^a\dot{x}^b)\partial_{x^c},
$$
where $\Gamma_{ab}^c(x)$ are the Christoffel symbols of the metric $g$ and $\dot{x}^a$ (resp. $\ddot{x}^a$) are
(resp. second-order) derivatives of $x^a(t)$ at $t=0$.
Therefore \eqref{e:Riemanniancubic} is indeed a function on $\T^2M$.

We shall now compute the second-order EL equations associated with the Lagrangian $L$ given by \eqref{e:Riemanniancubic}. To this end, recall two fundamental properties of the Levi-Civita connection:
\begin{align}
&\nabla_XY-\nabla_YX=[X,Y],\label{eqn:LC1}\\
&Xg(Y,Z)=g(\nabla_XY,Z)+g(Y,\nabla_XZ),\label{eqn:LC2}
\end{align}
which hold for any vector fields $X,Y,Z$ on the manifold $M$..

Note that although the Lie bracket $[X,Y]$ is defined for vector fields, to calculate its value at a point $p\in M$ it is enough to know vectors $X(Y)(p):=\T Y\circ X(p)\in\T_{Y(p)}\T M$ and $Y(X)(p)\in \T_{X(p)}\T M$ (see, e.g., \cite{Kolar_Michor_Slovak_nat_oper_diff_geom_1993}). According, we introduce the following notion: for a vector $X\in\T_pM$ and a vector $A\in\T_{X}\T M$ lying over $Y:=\T\tau_M(A)$ by an \emph{$A$--extension of $X$ around $p$} we will understand any (local) vector field $\wt X$ on $M$ such that $\wt X(p)=X$ and $Y(\wt X)(p)=A$. This means that $A$ is tangent to the graph of $\wt X$ at $X$.

Consider now a curve $\gamma(t)\in M$ and any generator $\del \gamma(t)\in\T_{\gamma(t)}M$  of an admissible variation $\del\jet 2\gamma(t)=\kappa_{2}(\jet 2\del\gamma(t))$. Let $\wt{\dot\gamma}$ be any $\del(\jet 1\gamma)(t)$--, i.e., $\kappa(\jet 1\del\gamma(t))$--extension, in the aforementioned sense,  of $\jet 1\gamma(t)=\dot \gamma(t)$ and let $\wt{\del\gamma}$ be any $\jet 1\del\gamma(t)$--extension of $\del \gamma(t)$ along $\gamma(t)$ (in particular $\wt{\del\gamma}=\del\gamma$ along $\gamma(t)$). It follows immediately from \cite{Kolar_Michor_Slovak_nat_oper_diff_geom_1993} (or \cite{MJ_MR_higher_algebroids_2013}, Proposition 2.2) that
\begin{equation}\label{eqn:bracket}
[\wt {\del\gamma},\wt{\dot\gamma}]=0\quad \text{along $\gamma(t)$}.
\end{equation}

We shall now compute the differential of the action $S_L$ in the direction of $\del\jet 2\gamma(t)$. Note that $g(\nabla_{\wt{\dot\gamma}}\wt{\dot\gamma},\nabla_{\wt{\dot\gamma}}\wt{\dot\gamma})$ is a function on $M$ coinciding with $L(\jet 2\gamma)=g(\nabla_{\dot\gamma}\dot\gamma,\nabla_{\dot\gamma}\dot\gamma)$ along $\gamma(t)$, hence
\begin{align*}
\<\dd S_L(\jet 2\gamma),\del\jet 2\gamma>=&\int_{t_0}^{t_1}\del\gamma g(\nabla_{\wt{\dot\gamma}}{\wt{\dot\gamma}},\nabla_{\wt{\dot\gamma}}{\wt{\dot\gamma}})\dd t\overset{\eqref{eqn:LC2}}=2\int_{t_0}^{t_1} g(\nabla_{\del\gamma}\nabla_{\wt{\dot\gamma}}{\wt{\dot\gamma}},\nabla_{\dot{\gamma}}{\dot{\gamma}})\dd t=\\
&2\int_{t_0}^{t_1} g\left(R({\del\gamma},{\dot{\gamma}}){\dot{\gamma}},\nabla_{\dot{\gamma}}{\dot{\gamma}}\right)+
g\left(\nabla_{\dot{\gamma}}\nabla_{\wt{\del\gamma}}{\wt{\dot\gamma}}+\nabla_{[\wt{\del\gamma},{\wt{\dot\gamma}}]}{\wt{\dot\gamma}},\nabla_{\dot{\gamma}}{\dot{\gamma}}\right)\dd t,
\intertext{where $R(X,Y)Z=\nabla_X\nabla_YZ-\nabla_Y\nabla_XZ-\nabla_{[X,Y]}Z$ is the Riemann tensor of $g$. By \eqref{eqn:bracket} and by the standard symmetry property $g(R(X,Y)Z,T)=g(R(T,Z)Y,X)$, the later equals}
&2\int_{t_0}^{t_1} g\left(R(\nabla_{\dot{\gamma}}{\dot{\gamma}},{\dot{\gamma}}){\dot{\gamma}},{\del\gamma}\right)+
g\left(\nabla_{\dot{\gamma}}\nabla_{\wt{\del\gamma}}{\wt{\dot\gamma}},\nabla_{\dot{\gamma}}{\dot{\gamma}}\right)\dd t.
\end{align*}
Let us transform the last integrand as follows
\begin{align*}
g&\left(\nabla_{\dot{\gamma}}\nabla_{\wt{\del\gamma}}{\wt{\dot\gamma}},\nabla_{\dot{\gamma}}{\dot{\gamma}}\right)
\overset{\eqref{eqn:LC1}\text{, }\eqref{eqn:bracket}}= g\left(\nabla_{\dot{\gamma}}\nabla_{\dot{\gamma}}\wt{\del\gamma},\nabla_{\dot{\gamma}}{\dot{\gamma}}\right)\overset{\eqref{eqn:LC2}}=\\
&\dot\gamma g\left(\nabla_{\dot\gamma}\del\gamma,\nabla_{\dot\gamma}\dot\gamma\right)-g\left(\nabla_{\dot\gamma}\del\gamma,\nabla_{\dot\gamma}\nabla_{\dot\gamma}\dot\gamma\right)
\overset{\eqref{eqn:LC2}}=
\dot\gamma g\left(\nabla_{\dot\gamma}\del\gamma,\nabla_{\dot\gamma}\dot\gamma\right)-\dot\gamma g\left(\del\gamma,\nabla_{\dot\gamma}\nabla_{\dot\gamma}\dot\gamma\right)+g\left(\del\gamma,\nabla_{\dot\gamma}\nabla_{\dot\gamma}\nabla_{\dot\gamma}\dot\gamma\right).
\end{align*}
This gives us
\begin{equation*}\begin{split}\<\dd S_L(\jet 2\gamma),\del\jet 2\gamma>=
&2\int_{t_0}^{t_1} g\left(R(\nabla_{\dot{\gamma}}{\dot{\gamma}},\dot\gamma)\dot\gamma+\nabla_{\dot\gamma}\nabla_{\dot\gamma}\nabla_{\dot\gamma}\dot\gamma,{\del\gamma}\right)\dd t+\\
&2\left[g\left(\nabla_{\dot\gamma}\del\gamma,\nabla_{\dot\gamma}\dot\gamma\right)-
g\left(\del\gamma,\nabla_{\dot\gamma}\nabla_{\dot\gamma}\dot\gamma\right)\right]\Bigg|_{t_0}^{t_1}.
\end{split}\end{equation*}
We see that the boundary term depends only on $\jet 1\del\gamma$, hence comparing the last expression with \eqref{eqn:action_var_full} we get that $\Upsilon_{2,\tau^\ast}\left(\jet 2\Lambda_L(\jet 2\gamma)\right)=g\left(R(\nabla_{\dot\gamma}\dot\gamma,\dot\gamma)\dot\gamma+\nabla_{\dot\gamma}\nabla_{\dot\gamma}\nabla_{\dot\gamma}\dot\gamma,\cdot\right)$. Consequently,  the EL equations for $L$ read
\begin{equation}\label{e:ELforRcp}
D_t^3 \dot{\gamma}(t) + R(D_t \dot{\gamma}(t), \dot{\gamma}(t)) \dot{\gamma}(t) =0,
\end{equation}
in agreement with \cite{NHP_splines_surfaces_1989}.
Solutions of \eqref{e:ELforRcp} are called \emph{Riemannian cubic polynomials}.


\appendix
\newpage
\section{Appendix}\label{app:geom_lemma}

\begin{proof}[Proof of Lemma~\ref{lem:Upsilon_canonical}]
Let us explain that the functoriality of a vector bundle morphism $\Thol{k} E\ra E$  means that actually we have a family $F=\{F_E\}$ of vector bundle morphisms
 as in \eqref{e:FE}, parameterized by vector bundles $\sigma:E\to M$, such that
 for any morphism
 $f: E_1\to E_2$ between vector bundles $\sigma_i: E_i\to M_i$, $i=1,2$, we have
 \begin{equation}\label{e:functor}
 F_{E_2}\circ \Thol{k} f =f\circ F_{E_1}.
\end{equation}

We shall derive the local coordinate form of $F_E$.
Observe first that the base map of $F_E$, denoted by $\und{F_E}: \TT{2k} M\to M$, has to be functorial as well. In other words $\und{F_E}$ is a natural transformation between Weil functors $\TT{2k}$ and $\operatorname{Id}$. It follows from \cite{Kolar_weil_bund_gen_jet_spaces_2008} that such a transformation corresponds to a unique homomorphism between the corresponding Weil algebras, namely
$h:\Weil^{2k}=\R[\nu]/\<\nu^{2k+1}>\ra\R$ given by $1\mapsto 1$ and $\nu \mapsto 0$. We conclude that $\und{F_E} = \tau^{2k}$ is the canonical bundle projection.

Every vector bundle can be locally described as $E=M\times V$, where $M$ is the base and $V$ is the model fiber. In this case  $\Thol{k}E = \TT{2k} M \times\TT k\TT kV$ and hence $F_E$ must be of the form
$F_E(v^{2k},X)=(\tau^{2k}(v^{2k}), L_{v^{2k}}(X))$, where $L_{v^{2k}}:\TT k\TT kV\ra V$ is a linear map, $v^{2k}\in \TT{2k} M$ and $X\in \TT k\TT kV$. Consider now a morphisms of the form $f=f_0\times\id_V: M\times V\to M\times V$,
where $f_0:M\to M$ is a smooth function. It follows from \eqref{e:functor} that $L_{v^{2k}}$ does not depend on $v^{2k}$ and hence, locally, $F_E$ is of the form
$$F_E(v^{2k},X)=(\tau^{2k}(v^{2k}), L(X)),$$
where $L: \TT k\TT k V\to V$ is a linear map.

Now we shall find the coordinate form of $L$. Since every vector space $V$ is a vector bundle over a one-point base, it follows that $F$ induces a functorial morphism
 $F_V: \Thol k V = \TT k\TT kV \to V$. But  $\TT k \TT kV \simeq  \Weil^{k, k}\otimes V$, canonically, where $\Weil^{k,k}$ denotes  the Weil algebra $\Weil^{k, k}=\R[\nu,\nu']/\<\nu^{k+1}, {\nu'}^{k+1}>$. Any
 functorial linear map $\Weil^{k,k}\otimes V \to V$ is of the form $f\otimes \id_{V}$ for some fixed linear map $f:\Weil^{k, k}\to \R$. If we denote $c_{\alpha\beta}:=f(\nu^\alpha{\nu'}^\beta)$, then we conclude that  general local form of $F_E:\Thol{k}E \to E$
is
$$
F_E\left(x^{a,(r)}, y^{i,(\alpha,\beta)}\right) = \left(\und{x^a}=x^{a,(0)}, \und{y^i}= \sum_{0\leq \alpha,\beta\leq k } c_{\alpha\beta} y^{i,(\alpha,\beta)}\right),
$$
where $(x^{a,(r)}, y^{i,(\alpha, \beta)})$ are the adapted coordinates on $\Thol{k}E$ (as defined in Preliminaries)
 induced from the standard coordinates $(x^a, y^i)$ on $E$ and we have underlined the coordinates in the co-domain.

To get more information on the coefficients $c_{\alpha\beta}$ it will be enough to consider the case $E= M\times \R$ with $V=\R$.
Consider the vector bundle morphism by $\wt{\phi}: M\times\R\to M\times\R$ given by
 $\wt{\phi}(x, y) = (x, \phi(x) \cdot y)$, where  $\phi \in C^\infty(M)$. We are looking for possible $F_E$ such that
 the following diagram
\begin{equation}\label{e:FE2}
\xymatrix{
\Thol{k} E \ar[d]_{\Thol{k} \wt{\phi}} \ar[rr]^{F_E}  && E \ar[d]^{\wt{\phi}} \\
\Thol{k} E  \ar[rr]^{F_E}  && E
}
\end{equation}
commutes. In our case, $\Thol{k} E = \TT{2k}\R\times \TT k\TT k\R$ and the commutativity of \eqref{e:FE2} reads as
\begin{equation}\label{e:ckl}
\sum_{0\leq \alpha,\beta\leq k} c_{\alpha\beta} \und{y}^{(\alpha,\beta)} = \phi(x) \cdot \sum_{0\leq \alpha,\beta\leq k} c_{\alpha\beta} y^{(\alpha,\beta)},
\end{equation}
where $\left(x^{(r)}, y^{(\alpha,\beta)}\right) \mapsto \left(x^{(r)}, \und{y}^{(\alpha,\beta)}\right)$ is the coordinate expression of the
morphism $\Thol{k}\wt{\phi}$. To obtain the coordinate expression of
$\TT k \TT k \wt{\phi}: \TT k\TT k E \to \TT k\TT kE$, for $E=M\times\R$ and  $M=\R$, which sends the class $[\gamma]$
of a map $\gamma:\R\times\R\to E$ to the class $[\wt\phi\circ \gamma]$ in  $\TT k\TT k E$, we write
\begin{align}\label{e:xykl}
x(\wt{\phi}(\gamma(s, t)))  &= x(\gamma(s, t)) =  \sum_{0\leq \alpha,\beta\leq k} x^{(\alpha,\beta)}([\gamma]) \frac{s^\alpha t^\beta}{\alpha! \beta!} + o(s^k, t^k),\\ \notag
y(\wt{\phi}(\gamma(s, t))) &= \phi(x(\gamma(s,t))) y(\gamma(s,t)) =\\\label{eqn:y}
&\left(\phi(x) +  \sum_{r=1}^r \frac{\phi^{(r)}(x)}{r!}h^r + o(h^k)\right)
\left(\sum_{0\leq \alpha, \beta\leq k} y^{(\alpha,m)}([\gamma])\frac{s^\alpha t^\beta}{\alpha!\beta!} + o(s^k, t^k)\right),
\end{align}
where $x=x(\gamma(0,0))=x^{(0,0)}([\gamma])$ and
$h= x(\gamma(s,t)) - x(\gamma(0,0))$ as in \eqref{e:xykl}. Clearly,
$\und{y}^{(\alpha,\beta)}$ is the coefficient of $s^\alpha t^\beta/\alpha!\beta!$ in $y(\wt{\phi}(\gamma(s, t)))$. For example, in case $k=2$
we find that
\begin{equation}\label{eqn:y21_underlined}
\und{y}^{(2,1)}  = y^{(2,1)}\phi(x) + 2 y^{(1,1)} x^{(1,0)}\phi'(x) + y^{(2,0)}x^{(0,1)}\phi'(x) + \ldots
\end{equation}
For any $l,m\geq 0$, $\TT l\TT m E$ is a three-fold graded bundle with bases
$E$, $\TT lE$ and $\TT mE$ \cite{JG_MR_gr_bund_hgm_str_2011}. Its algebra of multi-homogeneous functions, which is a subalgebra
of all smooth functions on $\TT l\TT m E$, is $\Z^3$-graded. For example, $y^{(\alpha,\beta)}$ is of degree $1$
with respect to the vector bundle structure over $\TT l\TT m M$ and of degrees $\alpha$ and $\beta$ with respect to
the bundle structures over $\TT m E$ and
$\TT lE$, respectively. Of course, $\TT k\TT k\wt{\phi}$ preserves this grading.
Restricting $\TT k\TT k\wt{\phi}$ to $\Thol kE$ means just replacing $x^{(\alpha,\beta)}$ with $x^{(\alpha+\beta)}$.
Thus we get
$$
\und{y}^{(\alpha,\beta)} = y^{(\alpha,\beta)} \phi(x) + \alpha y^{(\alpha-1,\beta)} x^{(1)}\phi'(x) + \beta y^{(\alpha, \beta-1)} x^{(1)} \phi'(x) + \ldots
$$
where, in case $\alpha=0$ or $\beta=0$, it is enough to put  $0$ instead of $y^{(-1,\beta)}$ or $y^{(\alpha, -1)}$.
By comparing the  $\phi'(x)$ coefficients in the above expression and separating the terms with respect to the gradation we find that
the necessary condition for \eqref{e:ckl} is
$$
\sum_{0\leq \alpha,\beta\leq k, \alpha+\beta = s} c_{\alpha\beta} (\alpha y^{(\alpha-1, \beta)} + \beta y^{(\alpha, \beta-1)}) = 0,
$$
for any $0\leq s\leq 2k$, with the convention that $y^{(\alpha, \beta)}=0$ whenever $\alpha$ or $\beta$ is negative or greater than $k$.
 This gives a recurrence relation between the coefficients $c_{\alpha\beta}$ when the sum $\alpha+\beta=s$ fixed. Namely,  $(\alpha+1) c_{(\alpha+1)(\beta-1)}+ \beta c_{\alpha\beta} = 0$ if $0\leq \alpha\leq k-1$ and $1\leq \beta\leq k$. Moreover, $c_{\alpha k}=0$ for $1\leq \alpha\leq k$. It follows that
for $0\leq s\leq k$ the vector $(c_{s0}, c_{1(s-1)}, \ldots, c_{0s})$ is proportional
to $(\binom{s}{0}, -\binom{s}{1}, \binom{s}{2}, \ldots, \pm \binom{s}{s})$
and $c_{\alpha\beta} =0$ whenever $\alpha+\beta>k$. Hence,
$$
\sum_{\alpha,\beta} c_{\alpha\beta} y^{(\alpha,\beta)} = \sum_{s=0}^k a_s \sum_{\alpha+\beta=s}(-1)^\alpha \binom{s}{\alpha} y^{(\alpha,\beta)}.
$$
for some coefficients $a_0, \ldots, a_k\in \R$. Therefore,  $F_E$ must have the desired form being a linear combination of morphisms $\Upsilon_{s, \sigma}$ given locally by ~\eqref{eqn:Upsilon_local}.
\end{proof}

A slight modification of the proof above shows that $\Upsilon_{k,\sigma}$ has no canonical extension to $\TT k\TT kE$.

\begin{cor} Any functorial vector bundle morphism
\begin{equation}\label{e:FE1}
\xymatrix{
\tgT^k\tgT^k E \ar[d]\ar[r]^{F_E}  & E \ar[d] \\
\tgT^k\tgT^k M \ar[r]^{\und{F_E}}  & M
}
\end{equation}
is proportional to $\tau_E^{(k,k)}:\tgT^k\tgT^k E\ra E$  covering $\tau^{(k,k)}:\tgT^k\tgT^k M\ra M$.
In particular, $\Upsilon_{k, \sigma}$ has, in general, no canonical extension to a functorial vector bundle morphism on $\tgT^k\tgT^k E$.
\end{cor}
\begin{proof}
Just go over the same reasoning of the proof of Lemma \ref{lem:Upsilon_canonical}. By considering a trivial vector bundle $E=M\times V$ and
endomorphisms of the form $f=f_0\times \id_V:E\to E$, where $f_0\in C^\infty(M)$,
we find that a general local form of $F_E:\tgT^k\tgT^k E\ra E$ has to be
$$
F_E(x^{a, (\alpha',\beta')}, y^{i,(\alpha,\beta)}) = (\und{x}^a = x^{a, (0,0)}, \und{y}^i = \sum_{0\leq \alpha,\beta\leq k} c_{\alpha\beta}y^{i, (\alpha,\beta)}).
$$
In the same way we find that the coefficients $c_{\alpha\beta}$ have to satisfy \eqref{e:ckl}, where
$(x^{(\alpha',\beta')}, y^{(\alpha, \beta)})\mapsto (x^{(\alpha',\beta')}, \und{y}^{(\alpha, \beta)})$ is the coordinate expression of the morphism
$\TT k\TT k\wt{\phi}$, where $\wt{\phi}:\R\times\R\ra\R\times\R$, $(x, y)\mapsto (x, \phi(x)\cdot y)$
for some function $\phi\in C^\infty(\R)$. In the transformation expression for $\und{y}^{(\alpha, \beta)}$ we separate terms with respect to the grading
and collect the terms which contain $\phi'(x)$. Then we find easily that \eqref{e:ckl} implies that $c_{\alpha,\beta}=0$ unless $\alpha=\beta=0$.
Therefore $F_E = c_{00} \cdot \tau_E^{(k, k)}$, as  claimed.
\end{proof}

\newpage
\section*{Acknowledgments}
This research was supported by Polish National Science Center grant
under the contract number DEC-2012/06/A/ST1/00256.

The authors are grateful for professors Paweł Urbański and Janusz Grabowski for reference suggestions. We wish to thank especially the second of them for reading the manuscript and giving helpful remarks. We are also grateful for the reviewers for many helpful suggestions.



\begin{thebibliography}{99}

\bibitem{Boor_splines_1978} C. de Boor, "A Practical guide to splines", Springer-Verlag, New York (1978).

\bibitem{Cantrin_Crampin_Inn_can_isom_1989} F. Cantrijn, M. Crampin, W.	Sarlet, D. Saunders,	The canonical isomorphism between $\TT k\TT\ast M$ and $\TT\ast \TT kM$, \emph{C. R. Acad. Sci. Paris} \textbf{309} (1989), pp. 1509--1514.

\bibitem{Crampin_EL_higher_order_1990} M. Crampin, Lagrangian submanifolds and the Euler-Lagrange equations in higher-order Mechanics, \emph{Lett. Math. Phys.} \textbf{17} (1990), pp. 53--58.		

\bibitem{Crampin_Sar_Cart_high_ord_diff_eqns_1986} M. Crampin, W.	Sarlet.	F. Cantrijn,	Higher-order differential equations and higher-order Lagrangian mechanics, \emph{Math. Proc. Cambridge Phillos. Soc.} \textbf{99} (1986), pp. 565--587.

\bibitem{Gay_Holm_Inn_inv_ho_var_probl_2012} F. Gay-Balmaz, D. D. Holm, D. M. Meier, T. S. Ratiu, F. Vialard, Invariant Higher-Order Variational Problems,
\emph{Comm. Math. Phys.} \textbf{309} (2012), pp. 413--458.

\bibitem{KG_private} K. Grabowska, private communication (2012).

\bibitem{KG_JG_var_calc_alg_2008} K. Grabowska, J. Grabowski, Variational calculus with constraints on general algebroids, \emph{J. Phys. A: Math. Theor.} \textbf{41} (2008), 175204.


\bibitem{JG_MR_gr_bund_hgm_str_2011} J. Grabowski, M. Rotkiewicz, Graded bundles and homogeneity structures, \emph{J. Geom. Phys.} \textbf{62} (2011), pp. 21--36.

\bibitem{Gracia_Martin_Munos_2003}\vskip-.2cm X. Gracia, J. Martin-Solano, M. Munoz-Lecenda, Some geometric aspects of variational calculus in constrained systems, \emph{Rep. Math. Phys.} {\bf 51} (2003), pp. 127--147.

\bibitem{MJ_MR_higher_algebroids_2013} M. Jóźwikowski, M. Rotkiewicz, Higher algebroids with application to variational calculus, preprint arXiv:1306.3379 (2013).

\bibitem{MJ_WR_nh_vac_2013} M. Jóźwikowski, W. Respondek, A comparison of vakonomic and nonholonomic variational problems with applications to systems on Lie groups, preprint arXiv:1310.8528 (2013).

\bibitem{Kolar_weil_bund_gen_jet_spaces_2008} I. Kolar, Weil Bundles as Generalized Jet Spaces, in: "Handbook of global analysis", Elsevier, Amsterdam (2008), pp. 625--664.

\bibitem{Kolar_Michor_Slovak_nat_oper_diff_geom_1993} I. Kolar, P.W. Michor, J. Slovak, "Natural operations in differential geometry", Springer, Berlin (1993).

\bibitem{Leon_Lacomb_lagr_sbmfd_ho_mech_sys_1989} M. de Leon, E. Lacomba, Lagrangian submanifolds and higher-order mechanical systems, \emph{J. Phys. A} \textbf{22} (1989), pp. 3809--3820.

\bibitem{Leon_Rodrigues_higher_almost_tangent} M. de Leon, P. R. Rodrigues,  Higher order almost tangent geometry and non-autonomous Lagrangian dynamics, in: "Proceedings of the Winter School 'Geometry and
Physics'", Circolo Matematico di Palermo, Palermo (1987), pp. 157-71.

\bibitem{Mackenzie_lie_2005} K. Mackenzie, General Theory of Lie groupoids and Lie Algebroids, CUP, Cambridge (2005).

\bibitem{Morimoto_Lifts} A. Morimoto, Liftings of tensor fields and connections to tangent bundles of higher order, \emph{Nagoya Math. J.} \textbf{40} (1970), 99--120. 


\bibitem{NHP_splines_surfaces_1989} L. Noakes, G. Heinzinger, B. Paden, Cubic splines on curved surfaces,
\emph{IMA J. Math. Control Inform.}, \textbf{6} (1989), pp. 465--473.

\bibitem{Saunders_Geom_of_jet_bndls} D. J. Saunders, "The geometry of jet bundles", \emph{Lecture Notes Math.} \textbf{142}, CUP (1989). 

\bibitem{Tulcz_diff_lagr_1975}	W.  Tulczyjew, Sur la diff\'{e}rentiele de Lagrange, \emph{C. R. Acad. Sci. Paris} \textbf{280} (1975), pp. 1295--1298.

\bibitem{Tulcz_lagr_diff_1976}	W.  Tulczyjew, The Lagrange differential, \emph{Bull. Acad. Polon. Sci.} \textbf{24} (1976), pp. 1089--1097.

\bibitem{Tulcz_dyn_ham_1976}	W.  Tulczyjew, Les sous-vari\'{e}t\'{e}s lagrangiennes et la dynamique hamiltonienne, \emph{C. R. Acad. Sci. Paris} \textbf{283} (1976), pp. 15--18.

\bibitem{Tulcz_dyn_lagr_1976}	W.  Tulczyjew, Les sous-vari\'{e}t\'{e}s lagrangiennes et la dynamique lagrangienne, \emph{C. R. Acad. Sci. Paris} \textbf{283} (1976), pp. 675--678.

\bibitem{Tulcz_ehres_jet_theory_2006}	W.  Tulczyjew, Evolution of Ehresmann's jet theory, in: "Geometry and topology of manifolds: the mathematical legacy of Charles Ehresmann", Banach Centre Publications \textbf{76}, Warsaw (2007), pp. 159--176.	

\bibitem{Viatgliano_lagr_ham_h_field_theor_2010} L. Vitagliano, The Lagrangian-Hamiltonian formalism for higher-order field theories, \emph{J. Geom. Phys.} \textbf{60} (2010), pp. 857-–873.

\bibitem{Weil_theo_point_proches_1953} A. Weil, Th\'{e}orie des points proches sur les varietes diff\'{e}rentiables, in: "Colloque de g\'{e}ométrie diff\'{e}rentielle", CNRS,
Strasbourg (1953), pp. 111--117.

\end{thebibliography}
\end{document}